\documentclass[12pt]{amsart}%%[12]

\usepackage{titletoc}
\usepackage{diagbox}
\usepackage{imakeidx}
\makeindex
\usepackage{pdfpages}
\usepackage{tikz}
\usepackage[colorlinks=true, linkcolor=blue, anchorcolor=blue, citecolor=blue, filecolor=blue, menucolor= blue, urlcolor=blue,pdfencoding=auto, psdextra]{hyperref}
\usepackage{bookmark}% faster updated bookmarks
\usepackage{amsmath,amssymb,amsthm,amscd}
\usepackage{verbatim}
\usepackage{comment}
\usepackage{multirow}
\usepackage{mathtools}
\usepackage{enumitem} 
\usepackage{pgf,tikz,mathrsfs}
\usetikzlibrary{arrows}
\usepackage[normalem]{ulem}
\usepackage{multicol}
\usepackage{xfrac,xcolor}
\usepackage[margin=1in]{geometry}
\usepackage{graphicx}
\numberwithin{equation}{section}
\usepackage{pstricks-add,pst-plot,pst-func}
\usepackage{cleveref}
\newtheorem{theorem}{Theorem}[section]
\newtheorem{proposition}[theorem]{Proposition}
\newtheorem{lemma}[theorem]{Lemma}
\newtheorem{corollary}[theorem]{Corollary}

\newtheorem{prob*}{Problem}
\newtheorem{question}[theorem]{Question}

\newtheorem{conjecture}[theorem]{Conjecture}
\newtheorem*{theorem*}{Theorem}

\theoremstyle{definition}
\newtheorem{definition}[theorem]{Definition}

\newtheorem{example}[theorem]{Example}
\newtheorem{remark}[theorem]{Remark}

\numberwithin{equation}{section}

\definecolor{green}{rgb}{.1,.75,.1}

\usepackage{xparse}
\usepackage{tikz}
\usetikzlibrary{matrix,backgrounds}
\pgfdeclarelayer{myback}
\pgfsetlayers{myback,background,main}

\tikzset{mycolor/.style = {line width=1bp,color=#1}}%
\tikzset{myfillcolor/.style = {draw,fill=#1}}%

\NewDocumentCommand{\highlight}{O{blue!40} m m}{%
	\draw[mycolor=#1] (#2.north west)rectangle (#3.south east);
}

\NewDocumentCommand{\fhighlight}{O{blue!40} m m}{%
	\draw[myfillcolor=#1] (#2.north west)rectangle (#3.south east);
}

\newcommand{\ZZ}{ \ensuremath{\mathbb{Z}}}

\newcommand{\PP}{ \ensuremath{\mathbb{P}}}

\renewcommand{\labelenumi}{\arabic{enumi}.  }
\renewcommand{\labelenumii}{(\alph{enumii})}

\definecolor{MyDarkGreen}{cmyk}{0.7,0,1,0}

\def\cocoa{{\hbox{\rm C\kern-.13em o\kern-.07em C\kern-.13em o\kern-.15em A}}}

\usepackage{etoolbox}
\patchcmd{\subsection}{-.5em}{.5em}{}{}
\patchcmd{\subsection}{2}{3}{}{}
\makeatletter
\newsavebox\myboxA
\newsavebox\myboxB
\newlength\mylenA

\newcommand*\xoverline[2][0.75]{%
    \sbox{\myboxA}{$\m@th#2$}%
    \setbox\myboxB\null% Phantom box
    \ht\myboxB=\ht\myboxA%
    \dp\myboxB=\dp\myboxA%
    \wd\myboxB=#1\wd\myboxA% Scale phantom
    \sbox\myboxB{$\m@th\overline{\copy\myboxB}$}%  Overlined phantom
    \setlength\mylenA{\the\wd\myboxA}%   calc width diff
    \addtolength\mylenA{-\the\wd\myboxB}%
    \ifdim\wd\myboxB<\wd\myboxA%
       \rlap{\hskip 0.5\mylenA\usebox\myboxB}{\usebox\myboxA}%
    \else
        \hskip -0.5\mylenA\rlap{\usebox\myboxA}{\hskip 0.5\mylenA\usebox\myboxB}%
    \fi}
\makeatother
\begin{document}

\begin{abstract} Let $C\subset \mathbb P^3$ be a curve over an algebraically closed field of characteristic zero, and let $M(C)$ denote its Hartshorne-Rao module. We study how the geometry of $C$ influences whether $M(C)$ satisfies the Weak and Strong Lefschetz Properties. We first consider unions of general skew lines and prove that multiplication by $L^i$, for a general linear form $L$, has maximal rank on $M(C)$ for $i=1,2,3$. The proof uses a specialization to zero-dimensional schemes that can be written as a union of curvilinear schemes, each of a particular type and of degree at most three, together with generic Hilbert function results for such schemes, which are of independent interest. We then examine how special geometric configurations can affect the Weak Lefschetz Property. In particular, we show that curves on a smooth quadric surface have Hartshorne-Rao modules with the Weak Lefschetz Property, and that the property persists for unions of skew lines with all but one line on a quadric. By contrast, for $r\geq 10$, we construct configurations of $r$ skew lines with all but two lines on a quadric whose Hartshorne-Rao modules fail the Weak Lefschetz Property. Finally, we study smooth irreducible curves. We prove the Weak Lefschetz Property in several low-degree cases, construct a degree 15 curve for which it fails, and show that general nondegenerate rational curves have Hartshorne-Rao modules with the Weak Lefschetz Property. These results illustrate both the strength and the limitations of geometric hypotheses in controlling Lefschetz properties of Hartshorne-Rao modules.
\end{abstract}

%\frontmatter
\title[Lefschetz Properties for Hartshorne-Rao Modules of Curves in $\mathbb P^3$]{Weak and Strong Lefschetz Properties for Hartshorne-Rao Modules of Curves in $\mathbb P^3$}

\author[J.~Migliore]{Juan Migliore} 
\address[J.~Migliore]{Department of Mathematics,
University of Notre Dame,
Notre Dame, IN 46556, USA}
\email{migliore.1@nd.edu} 

\author[U. Nagel]{Uwe Nagel}
\address[U. Nagel]{Department of Mathematics, University of Kentucky, 715 Patterson Office Tower, Lexington,
KY 40506,  USA}
\email{uwe.nagel@uky.edu}

\author[C. Peterson]{Chris Peterson}
\address[C. Peterson]{Department of Mathematics, Colorado State University, Fort Collins, CO 80523, USA}
\email{christopher2.peterson@colostate.edu}

\author[E.T. Turatti]{Ettore Teixeira Turatti}
\address[E.T. Turatti]{Faculty of Mathematics, Informatics, and Mechanics, University of Warsaw, Banacha 2, 02-097 Warsaw, Poland}
\email{e.teixeira-turatti@uw.edu.pl}

\maketitle

\tableofcontents

\section{Introduction}

Let $\mathbb K$ be an algebraically closed field of characteristic zero. Let $R = \mathbb K[x_0,\dots,x_n]$ be a polynomial ring with its standard grading. By a curve we mean a locally Cohen-Macaulay one-dimensional subscheme of $\mathbb P^n$, equivalently a one-dimensional subscheme with no embedded or isolated points. For such a curve $C \subset \PP^n$ we define its Hartshorne-Rao module, $M(C)$, by
\[
M(C) = \bigoplus_{t \in \ZZ} H^1(\mathbb P^n,\mathcal I_C(t)).
\]
Since $C$ is an equidimensional one-dimensional subscheme of $\PP^n$, its Hartshorne-Rao module is a graded $R$-module of finite length (i.e. its dimension as a $\mathbb K$-vector space is finite).

Let $L \in R$ be a general linear form. Recall that a graded $R$-module $M$ of finite length has the Weak Lefschetz Property (WLP) if the multiplication map $\times L \colon [M]_{t-1} \rightarrow [M]_t$ has maximal rank for every integer $t$. It has the Strong Lefschetz Property (SLP) if the multiplication map $\times L^m \colon [M]_{t-m} \rightarrow [M]_t$ has maximal rank for all integers $t$ and all positive integers $m$. Note that $[M]_t$ denotes the degree $t$ component of the graded $R$-module $M$ and is a $\mathbb K$-vector space. By a slight abuse of language, we will often say $M$ has WLP (resp. SLP) rather than $M$ has {\it the WLP} (resp. {\it the} SLP). {Finally, we say that $M$ has SLP {\it in range $i$} if $\times L^i \colon [M]_{t-i} \rightarrow [M]_t$ has maximal rank for all integers $t$.} 

The goal of this paper is to study how the geometry of the curve $C$ influences whether either Lefschetz property holds for the  module $M(C)$. At first glance, one might guess that if $C$ is smooth and irreducible then $M(C)$ has WLP (for example, one might try to use the  fact that the general hyperplane section of $C$ has the Uniform Position Property \cite{harris} if $C$ is nondegenerate). However, this approach will not work.  It was shown in \cite{rao} that up to shift, any graded module of finite length is the Hartshorne-Rao module of some smooth curve. For instance, choose a nonzero finite-length graded module $M$ such that multiplication by every linear form is zero between two consecutive nonzero pieces. Then no general linear form can act with maximal rank in that degree, so $M$ fails WLP. Rao's theorem then gives a smooth curve $C$ whose Hartshorne-Rao module is, up to twist, $M$.
As a consequence, the smoothness of $C$ alone is not enough to draw any conclusions about WLP for the Hartshorne-Rao module $M(C)$, and indeed, we will exhibit in \Cref{11buchs} an example of a smooth connected curve for which WLP fails. In this paper, we will explore how additional information about the curve $C$ can force or preclude WLP for $M(C)$.

We note that if $C$ is a curve for which $M(C)$ has WLP or SLP, then the same is true for $M(C^\prime)$, where $C^\prime$ is any curve in the liaison class of $C$. This is clear for the even liaison class of $C$, since each of WLP and SLP for a module $M$ is invariant up to shift. However, it is also true for curves linked to $C$ in an odd number of steps since WLP and SLP are preserved when passing to the $\mathbb K$-dual of $M(C)$. See \cite{rao} for the relevant facts about the connections between $M(C)$ and the liaison class of $C$.

In this work, we focus first on configurations of lines, where the geometry can be varied in a controlled way, and then on smooth irreducible curves of low degree. We now summarize our main results.

In \Cref{sec: general lines}, we show that if $C\subset \mathbb P^3$ is a union of general lines then $M(C)$ has maximal rank for multiplication by $L^i$ for $i=1,2,3$; see \Cref{Prop:GenlWLP}, \Cref{SLP2}, and \Cref{SLP3}.  The proof proceeds by relating the multiplication maps to the Hilbert functions of zero-dimensional schemes that are unions of special nonreduced curvilinear schemes (which we refer to as flat fat points). A key result is \Cref{GenFlatFat}, which says that a general finite set of flat fat points, each with multiplicity $m \le 3$, has generic Hilbert function as a subscheme of a degree-$m$ hypersurface. Such sets arise in our argument when studying the Hilbert function of the intersection of a general finite set of lines in $\PP^3$ with $H^m$, the $m$-fold thickening of a plane, where $H^m$ is defined by $L^m$ with $L$ a general linear form. We expect that analogous genericity results are true for general unions of flat fat points of any fixed multiplicity $m \ge 4$. 

{After \Cref{sec: general lines}, we focus on WLP. We are interested in how specializing some of these lines may interfere with WLP for $M(C)$.  We consider, in \Cref{sec4}, configurations of skew lines in which many of the lines are constrained to lie on a quadric surface. We show that, in the extreme case, WLP holds when all lines lie on a quadric; see \Cref{prop:lines quadric}. We conclude that the Hartshorne-Rao module of any space curve on a quadric surface has WLP; see \Cref{curve on quadric}. Thus, it is natural to ask whether WLP also holds in less extreme cases. We show in \Cref{all but one} that when all but one of the lines lie on a quadric surface, then $M(C)$ continues to have WLP, but in Proposition \ref{all but 2}  we show that for $r \geq 10$ skew lines, if all but two lie on a quadric surface then $M(C)$ fails to have WLP.

In \Cref{sec5}, we consider the case of smooth irreducible curves. We show that $M(C)$ has WLP if $\deg(C)< 10$; see \Cref{at least 10}. We also exhibit an example of a smooth degree $15$ curve for which $M(C)$ fails WLP; see Proposition \ref{11buchs}. We further show that for every smooth curve $C$  of degree at most $14$, $M(C)$ has WLP, apart from three cases that remain open; see \Cref{< 15}. We close this section by showing that the Hartshorne-Rao module of a general nondegenerate rational curve of degree $d$ has WLP; see \Cref{prop:rational curve}. 
}

Some of the results in \Cref{sec5} rely on properties of finite sets of points with the uniform position property (UPP). These are well-known results that have been frequently used in the literature, but we are not aware of explicit proofs. We provide arguments in \Cref{sec:prelim} (\Cref{lem:irredinit} and \Cref{UPP cor}). 

\

\noindent {\bf Acknowledgments:} This work was begun during the Workshop on Weak and Strong Lefschetz Properties Across Mathematics, which  took place in Nordfjordeid, Norway on June 2--6, 2025; the  authors are grateful for the kindness and support of the Sophus Lie Center. We also thank Pietro De Poi, {\L}ucja Farnik, Brian Harbourne, Giovanna Ilardi, Jake Kettinger, Tomasz Szemberg and Justyna Szpond for their participation and contributions to the group where in which this project began.
Migliore was partially supported by Simons Foundation grant \#839618. Nagel was partially supported by Simons Foundation grant \#636513. Peterson was partially supported by NSF \#2428052. Turatti acknowledges support by the National Science Center, Poland, under the project “Tensor rank and its applications to signature tensors of paths”, 2023/51/D/ST1/02363.

%%%%%%%%%%%%%%%%%%%%%%%%%%%%%%%%%%%%%%%%%%%%%%%%%%%%%%%%%%%%%%%%%%%%%%%

\section{Useful exact sequences and properties of points in uniform position}
\label{sec:prelim}

In this section, we introduce notation and review the concepts and facts needed later. 

We will soon specialize the kind of curve we study, but for now denote by $C$ a curve that is an equidimensional, one-dimensional subscheme of $\PP^n$. Some of the sequences that follow were used in \cite{GeomInv} to address questions related to several of those considered in this paper. Our first exact sequence is an exact sequence of vector spaces; it is quite standard and will be used frequently.

\begin{equation} \label{usual 1}
0 \rightarrow [I_C]_t \rightarrow [R]_t \rightarrow H^0(\mathbb P^n, \mathcal O_C(t)) \rightarrow [M(C)]_t \rightarrow 0.
\end{equation}

Let $n=3$ and let $L \in [R]_1$ be a general linear form, defining a hyperplane $H$ in $\PP^3$. Let $m \geq 1$ be a positive integer. We want to study the effect on $M(C)$ of multiplication by $L^m$. As preparation for the next exact sequence, we introduce some notation.

Let $H^m$ denote the nonreduced surface defined by $L^m$ and let $Z_m$ be the subscheme of $\PP^3$ defined by $C \cap H^m$. It is defined by the saturation of the ideal $I_C + (L^m)$. Clearly $Z_m$ is a subscheme of $\PP^3$ and it is also a subscheme of $H^m$. We denote by $I_{Z_m}$ the homogeneous ideal of $Z_m$ in $R$ and by $I_{Z_m | H^m}$ the homogeneous ideal of $Z_m$ in $R/(L^m)$. These are related by the following exact sequence of modules
\begin{equation} \label{usual 2}
    0 \rightarrow R(-m) \rightarrow I_{Z_m} \rightarrow I_{Z_m | H^m} \rightarrow 0.
\end{equation}

If we sheafify, we get the following exact sequence
\begin{equation}
    0 \rightarrow \mathcal O_{\mathbb P^3}(-m) \rightarrow \mathcal I_{Z_m} \rightarrow \mathcal I_{Z_m| H^m} \rightarrow 0.
\end{equation}

Noting that $H^1(\mathcal O_{\mathbb P^3}(-m+t))=H^2(\mathcal O_{\mathbb P^3}(-m+t))=0$ for all $t$, from the associated long exact sequence in cohomology we can conclude that $H^1(\mathcal I_{Z_m}(t)) = H^1(\mathcal I_{Z_m|H^m}(t))$ for all $t$.

{Let $X\subset \PP^n$ be a scheme and let $H$ be a degree $d$ hypersurface. We have the Castelnuovo exact sequence \begin{equation}\label{eq:castelnuovo}
0\to \mathcal I_{X:H,\PP^n}(-d)\to \mathcal I_{X,\PP^n}\to \mathcal I_{X\cap H,H}\to 0,
\end{equation}
where $X:H$ is the scheme defined by the ideal $I_X:I_H$. If we denote $I_H=(F)$, at the level of global sections, the first map is multiplication by $F$.}

{Therefore, if $C \subset\PP^n$ is a curve, $H=V(L)$, where $L$ is a general linear form, and we consider the scheme $H^m$ defined by $L^m$, twisting by $\mathcal O_{\PP^n}(t)$ and taking the long exact sequence in cohomology associated to \eqref{eq:castelnuovo} yields:}

\begin{equation} \label{usual 3}
0 \rightarrow [I_C]_{t-m} \stackrel{\times L^m}{\longrightarrow} [I_C]_t \rightarrow [I_{Z_m | H^m}]_t \rightarrow [M(C)]_{t-m} \stackrel{\times L^m}{\longrightarrow} [M(C)]_t \rightarrow H^1(\mathcal I_{Z_m|H^m}(t)) \rightarrow \dots 
\end{equation}

The following definition is a slight extension of the original one, which assumed that $X$ and $Z$ are reduced.

\begin{definition}[\cite{GMR}]
    Let $X \subset \PP^n$ be a closed subscheme. Let $Z \subset X$ be a zero-dimensional scheme of degree $s$. Then $Z$ is in {\it generic position} on $X$ if its Hilbert function satisfies
    \[
    h_Z(t) = \min \{ h_X(t), s \}
    \]
    for all $t$.
\end{definition}

\begin{lemma} \label{h0h1=0}
    Let $X \subset \PP^3$ be a surface defined by a form $F$ of degree $d$. Let $Z \subset X$ be a zero-dimensional scheme in generic position on $X$. Then
    \[
    h^0(\mathcal I_{Z|X} (t)) \cdot h^1(\mathcal I_{Z|X} (t)) = 0.
    \]

    \begin{proof}
    Note that $I_X \cong R(-d)$ (in particular $X$ is arithmetically Cohen-Macaulay (ACM)).
As long as $h_Z(t) = h_X(t)$, we have $h^0(\mathcal I_{Z|X} (t)) = 0$, thanks to the exact sequence
\[
0 \rightarrow [R(-d)]_t \rightarrow [I_Z]_t \rightarrow [I_{Z|X}]_t \rightarrow 0.
\]
As soon as $h_Z(t) = s$, $h_Z$ has reached the multiplicity of $Z$, so $Z$ imposes independent conditions on forms of degree $t$ in $\PP^3$; hence $h^1(\mathcal I_Z(t)) = h^1(\mathcal I_{Z|X}(t)) = 0$ (from the same exact sequence above, sheafifying and taking cohomology).
    \end{proof}
\end{lemma}

For the next result, which we will need in \Cref{sec5}, we recall some notation. For a  finitely generated graded module $M$, we denote by $\alpha(M)$ the initial degree of $M$. Let $C$ be a curve with general hyperplane section $C \cap H = \Gamma$. Let $A$ be the coordinate ring of $\Gamma | H$ and let $\bar A$ be its general Artinian reduction. Let $a$ be the least degree of a socle element of $\bar A$ and let $b$ denote the least degree in which $I_{\Gamma|H}$ contains an element that does not lift to $I_C$.
Then \cite[Theorem 3.16]{HU} implies the following:  

\begin{theorem}[Huneke-Ulrich] \label{HU fact}
If $C$ is not ACM then $a \leq b$. That is, the least degree in which $I_{\Gamma|H}$  contains an element that does not lift to $I_C$ is greater than or equal to the least degree of a socle element of the general Artinian reduction of $A$.  
\end{theorem}

\begin{definition}
    A set of points $\Gamma\subset \PP^n$ is in General Linear Position (GLP) if every subset of $t\leq n+1$ points is linearly independent.
\end{definition}
To define the Uniform Position Property (UPP), we recall the following result by Harris.

\begin{lemma}[{\cite[Uniform Position Lemma]{harris}}]Let $C\subset \PP^n$ be a reduced, irreducible, and nondegenerate curve, and let $\Gamma$ be a general hyperplane section of $C$. If $\Gamma$ imposes $t$ independent conditions on $|\mathcal O_{\PP^{n-1}}(\ell)|$, then any subset of $t$ points of $\Gamma$ also imposes $t$ conditions. \end{lemma}

 We say a set of points has UPP if it satisfies the same property as $\Gamma$ in the previous lemma. 
Notice that any set in uniform position is in general linear position.

We will use the following sequence of facts in Section \ref{sec5}. They are very well known and have been used frequently without proof. Since we could not find a reference, we include arguments here for completeness.  (We note that \cite[Theorem 3.4]{GerMar} uses a concept of uniform position, which is different from the commonly used terminology defined above.)

\begin{lemma}
    \label{lem:irredinit}
Let $Z \subset \mathbb P^n$ ($n \ge 2$) be a finite set of points with the UPP. Denote by $a$ the initial degree of $I_Z$. Then a general form in $[I_Z]_a$ is irreducible.
\end{lemma}

\begin{proof}
We use induction on $a \ge 1$. If $a = 1$ the claim is clear. Assume $a \ge 2$.

 According to Bertini's theorem \cite[Theorems 5.1, 5.2]{kleiman}, there are two cases that we have to rule out. 
 
Case 1: Assume $[I_Z]_a$ is composite with a linear pencil $\mathcal{L}$. Denote the generators of $\mathcal{L}$ by $F$ and $G$. Thus, every element of $[I_Z]_a$ is a product of elements of the vector space $\langle F,G \rangle$. The factors of a general polynomial in $[I_Z]_a$ vanish precisely at the support of the subscheme defined by $F$ and $G$. Hence, each factor vanishes on $Z$ and its degree is less than $a$, which contradicts the definition of $a$.

Case 2: Assume the forms in $[I_Z]_a$ have a common divisor of positive degree. Denote by $t$ the maximum of the number of points in $Z$ vanishing on an irreducible common divisor of $[I_Z]_a$ with positive degree. Denote by $G$ one of these irreducible divisors that vanishes on $t$ points of $Z$. Put $d = \deg G$. If $d = a$ we are done because then it follows that $[I_Z]_a$ is generated by $G$. 

Assume $d < a$. Let $Z_1$ be the set of points of $Z$ lying on $G$ and put $Z_2 = Z \backslash Z_1$. Note that $t = |Z_1|$ and 
$I_{Z_2} = I_Z : G$ and that $\dim_\mathbb{K} [I_Z ]_a = \dim_\mathbb{K}[I_{Z_2} ]_{a-d}$. Since $1 \le d < a$, it follows that both $Z_1$ and $Z_2$ are not empty. Choose points $P_1 \in Z_1$ and $P_2 \in Z_2$ and set $Y_1 = Z_1 - P_1 + P_2$ and $Y_2 = Z_2 - P_2 + P_1$. 
By UPP there is a form $G' = G_{P_1, P_2}$ of degree $d$ vanishing on $Y_1$. Since $G$ does not vanish on $P_2$, the forms $G$ and $G'$ define distinct hypersurfaces. Using UPP again, we see that the vector spaces $[I_{Z_2}]_{a-d}$ and $[I_{Y_2}]_{a-d}$ have the same dimension. Hence, $G' \cdot [I_{Y_2}]_{a-d} \subseteq [I_{Z}]_{a}$ implies equality by comparing dimensions. Consider any $F \neq 0$ in $[I_{Y_2}]_{a-d}$. Since $G$ is irreducible with the same degree 
as $G'$ and $G$ is a divisor of $G' F \in I_Z$, the form $F$ factors as $F = G H$. By definition $G$ does not vanish on any point of $Z_2$ but $F$ vanishes on $Y_2$. Hence $H$ must vanish on $Z_2 - P_2$. Since $G'$ vanishes on $Y_1$ it follows that $G' H$ vanishes on $Z - P_1$. Choose a linear form $L$ in $I_{P_1}$. It follows that $G' H L$ is in $I_Z$ 
and has degree $a - d + 1$. If $d \ge 2$ then this is a contradiction to the definition of $a$ as initial degree. Thus, we must have $d =1$. So, $G' = G_{P_1, P_2}$ is an irreducible common divisor of $[I_Z]_a$. By definition of $t$, we get that $G_{P_1, P_2}$ vanishes at not more than $t  = |Z_1|$ points of $Z$. It follows that the set of these points is exactly $Y_1 = Z_1 -P_1 + P_2$. Hence,  any two distinct choices of pairs of points $(P_1, P_2)$ as above give distinct irreducible factors $G_{P_1, P_2}$ of $F$. 
Since $Z$ spans $\PP^n$ and $n \ge 2$, at least one of $Z_1$ or $Z_2$ has at least two elements. It follows that there is a pair $(Q_1, Q_2) \neq (P_1, P_2)$ with $Q_1 \in Z_1$ and $Q_2 \in Z_2$. Thus, $G, G' = G_{P_1, P_2}$ and $G'' = G_{Q_1, Q_2}$ are common divisors of $[I_Z]_a$ and linearly independent. 

Notice that the set $Z_2$ has the UPP and the initial degree of $I_{Z_2}$ is $a-1$. Hence, by induction on $a$, a general form $F$ in $I_{Z_2}$ is irreducible. Since $F G$ is in $I_Z$, it follows that the gcd of $[I_Z]_a$ is $FG$. This contradicts the fact that $G G'  G''$ divides the gcd of $[I_Z]_a$ and completes the proof. 
\end{proof}

The above result can be extended as follows. 

\begin{corollary}
\label{UPP cor}
    Let $Z$ be a finite set of points in $\mathbb P^n$ ($n \ge 2$) with the UPP. Denote by $a$ the initial degree of $I_Z$ and let $b$ be the least degree in which $\dim_{\mathbb K} [I_Z]_b > \binom{b-a +n}{n}$.  
For any integer $t$, a general element in $[I_Z]_t$ is irreducible if and only if $t = a$ or $t \ge b$. 
\end{corollary}

\begin{proof} 
If $t < a$ then $[I_Z]_t = 0$ and the claim is clear. 
If $b = a$ then we are done by \Cref{lem:irredinit}. 

If $a \le t < b$, the definition of $b$ gives that $ [I_Z]_a$ is generated by one form, $F$, and that $F$ is a gcd of $ [I_Z]_t$. Hence any element of $ [I_Z]_t$ is reducible. 

It remains to consider the case where $a < b \le t$. {Since the general element of $[I_Z]_a$ is irreducible,}
by definition of $b$ the forms in 
$[I_Z]_t$ do not have a common factor. Thus, we have to rule out that $[I_Z]_t$ is composite with a linear pencil 
$\mathcal{L}$. Consider a nonzero $F \in [I_Z]_a$ and a general form $G \in R$ of degree $t - a$. Then $G$ is irreducible, since $n \ge 2$, and is not in the pencil $\mathcal{L}$. Hence, $FG \in [I_Z]_t$ implies that $[I_Z]_t$ cannot be composite with the pencil $\mathcal{L}$. 
\end{proof}

We illustrate the result by an example. 

\begin{example}
    Let $\lambda$ be a line in $\mathbb P^3$ and consider the curve $C_1$ defined by $I_{C_1} = I_\lambda^3$. The curve $C_1$ is not reduced: it has degree 6 and is supported on $\lambda$. The linear system $|[I_{C_1}]_3|$ is composite with a pencil, since every element is a product of three linear forms in $I_\lambda$ (possibly repeated).

    Consider complete intersections of type $(5,5)$ in $I_{C_1}$. The general residual curve under such a link is never irreducible. Indeed, if $(F,G)$ is such a complete intersection then each of $F$ and $G$ is triple along $\lambda$, so the residual, $C_2$, contains the nonreduced curve defined by $I_\lambda^2$ as a component. This component has degree 3.

    Taking the colon ideal $I_{C_2} : I_\lambda^2$, we obtain a  curve $C_3$ of degree $16 = 25 - 6 - 3$. One can check on a computer that $C_3$ is smooth and lies on a pencil of quintics, namely the quintics that defined the link. 

    Now let $Z$ be a general set of 200 points on $C_3$. Note that $Z$ is a set of points with UPP. These points also have an ideal that starts in degree 5, and the corresponding linear system $|[I_Z]_5|$ is the same pencil. The base locus of this pencil is not reduced or irreducible, by construction. Furthermore, since the base locus contains curves defined by a power  of $I_\lambda$, the general element of this linear system is not smooth (although it is smooth away from the base locus). Nevertheless, the linear system is not composite with a pencil and does not have a 2-dimensional base locus, and so the general element is irreducible.
\end{example}

As a consequence of \Cref{UPP cor}, we get: 

\begin{corollary} \label{cor:HF of CI}
    Let $Z \subset \mathbb P^2$ be a set of points satisfying the UPP. If $Z$ has the Hilbert function of a complete intersection, then it is a complete intersection.
\end{corollary}

\begin{proof}
    Say $Z$ has the Hilbert function of a complete intersection of type $(a,b)$, $a \leq b$. A general element $F$ of degree $a$ is irreducible, and the Hilbert function forces $b$ to be as in Corollary \ref{UPP cor}. Since $F$ is irreducible, there is no common factor among the elements of degree $b$ so the result follows.
\end{proof}

%%%%%%%%%%%%%%%%%%%%%%%%%%%%%%%%%%%%%%%%%%%%%%%%%%%%%%%%%%%%%%%%%%%%%%%

\section{General sets of skew lines in $\mathbb P^3$}
\label{sec: general lines}

It has been known since \cite{HH} what the dimensions of the Hartshorne-Rao module are of a general set $C$ of skew lines. In this section we show that for such a curve, $M(C)$ has WLP and, at least in range $\leq 3$, the SLP. We begin with the former. 

\begin{proposition} \label{Prop:GenlWLP}
    Let $C$ be a general set of skew lines  in $\PP^3$. Then $M(C)$ has WLP.
\end{proposition}

\begin{proof}
    In the exact sequence (\ref{usual 3}), $Z_1|H$ just represents a general set of points on a plane, but viewed as being on the  plane instead of in $\PP^3$. Then the result follows from Lemma~\ref{h0h1=0} and standard facts about a general set of points in $\PP^2$.
\end{proof}

We expect that $M(C)$  even has the SLP. Below we establish partial results in that direction.  The technical heart of our argument is a specialization argument.  
We begin with a definition.

\begin{definition}
A {\it curvilinear scheme $A$ of multiplicity $m$ in $\mathbb P^n$ supported on a point $P$} is a scheme whose ideal has the form $[I_P^m,I_C]^{sat}$ where $C \subset \PP^n$ is a curve that is smooth at $P$. If $C$ is a line, we say that $A$ is a {\it flat fat point}. Its ideal is of the form $(\ell^m, I_C)$, where $\ell$ is any linear form such that $I_P = (\ell, I_C)$. 
\end{definition}

A flat fat point is determined by $P$, the tangent direction, and $m$. Its $h$-vector has the form $(1,1,\dots,1)$ where there are $m$ entries equal to $1$. As a result, it imposes independent conditions on forms of degree $\geq m-1$ but not in smaller degrees. A union of $s$ disjoint flat fat points of multiplicity $m$ has degree $ms$. We denote by $Z_{s,m}$ such a union if the points and the directions are chosen generically and $m$ is fixed.

\begin{remark} 
  \label{rem:flex}
Note that given a reduced curve $C \subset \mathbb P^2$  where no component is a line, a general point is not a flex \cite[Exercise 5.23]{FultonBook} (i.e. at a general point the tangent line meets $C$ with multiplicity $2$, not $3$). Thus imposing that a line meets $C$ with multiplicity 3 at a point is one additional condition beyond imposing tangency in general.    
\end{remark}

\begin{proposition} \label{GenFlatFat}
    Let $Z_{s,m}$ be a general union of $s$ flat fat points of multiplicity $m$ in $\mathbb P^2$. Assume that $1 \leq m \leq 3$. If $m = 1$ or $m=2$, assume $s \geq 1$, and if $m = 3$ assume that $s \geq 3$. 
    
    Then $Z_{s,m}$ has generic Hilbert function, i.e.
    \[    
    h_{Z_{s,m}}(j) = \min \left \{ ms, \binom{j+2}{2} \right \}.
    \]

\end{proposition}

\begin{proof}

First note that imposing vanishing at a flat fat point of multiplicity $m$ imposes at most $m$ conditions on a linear system. We focus on the question of whether these conditions are independent.

When $m=3$ and $s=1$ or $2$, the statement is not true. The case $s=1$ is immediate, and the case $s=2$ follows since $Z_{2,3}$ clearly lies on a conic, which is excluded by the claimed Hilbert function; hence the assumption $s \geq 3$ is necessary in the case $m=3$.

\vspace{.2in}

\noindent {\bf Case 1}: $m=1$.    

The case $m=1$ is well known. 

    \vspace{.2in}
    
\noindent {\bf Case 2}: $m=2$.   

For the case $m=2$, since specifying $Z_{s,2}$ is equivalent to specifying the points and the tangent directions, and since the points are general, it would be possible to apply \cite[Theorem 4.1]{CG}. However, we will give a different argument that extends to $m=3$.

We proceed by induction on $s \geq 1$. For $s=1,2$, the result is obvious. Assume $s=3$. The claimed Hilbert function is equivalent to the claim that $Z_{3,2}$ lies on no conic. Suppose it did. First of all, there is a pencil of conics containing the first two flat fat points, so imposing passage through the third general point, say $P_3$, (without yet imposing a tangent direction), gives a unique conic. However, this has a uniquely determined tangent direction, so imposing a general tangent direction through $P_3$ means there is no such conic. This forms the base case for the induction.

Now assume we have $Z_{s,2}$, a general union of $s$ flat fat points of multiplicity 2 with generic Hilbert function, and suppose we impose vanishing at a new general flat fat point. 
Let $\sigma$ be the initial degree of the ideal of $Z_{s,2}$. If $\dim [I_{Z_{s,2}}]_\sigma =1$ then imposing passage through a general point makes $\dim [I_{Z_{s,2}}]_\sigma =0$ and imposing the tangent direction gives a new condition in degree $\sigma+1$. Otherwise $\dim [I_{Z_{s,2}}]_\sigma \geq 2$ and we get two conditions in degree $\sigma$. This concludes the proof of Case 2.

\vspace{.2in}

\noindent {\bf Case 3}: $m=3$.

We recall from the statement of \Cref{GenFlatFat} that we now also have $s \geq 3$.

In this proof we will use Bertini's theorem, so we will first show that $[I_{Z_{s,3}}]_i$ is never composite with a pencil, and that it does not have a fixed common factor except in one simple situation.

Concerning common factors we have the following.

\vspace{.1in}

\noindent \underline{Claim}: 
    {\it Let $Z = Z_{s,3}$  be a general set of flat fat points in $\mathbb P^2$, supported on the general set of points $\{ P_1,\dots,P_s \}$. Let $t$ be the initial degree of $I_Z$. Then for any integer $i \geq t$,  the elements of $[I_Z]_i$ have a common factor, $G$, if and only if $\deg G = t$ and  $\dim [I_Z]_i = \binom{i-t+2}{2}$. (In particular, $\dim [I_Z]_t = 1$ and $I_Z$ has not yet reached its second generator in degree $i$.)}

\vspace{.1in}

Recalling that both $Z_{s,1}$ and $Z_{s,2}$ have generic Hilbert functions, and that $s \geq 3$, let
\[
\begin{array}{rclllll}
a & = & \hbox{initial degree of $I_{Z_{s,1}}$} & = & \min \{ i \ | \ \binom{i+2}{2} > s \} \geq 2, \\

b & = & \hbox{initial degree of $I_{Z_{s,2}}$} & = & \min \{ i \ | \ \binom{i+2}{2} > 2s \} \geq 3, \\

t & = & \hbox{initial degree of $I_Z = I_{Z_{s,3}}$} & \leq & \min \{ i \ | \ \binom{i+2}{2} > 3s \} .

\end{array}
\]

We prove the claim via some observations.

\begin{enumerate}

\item Clearly $a \leq b \leq t$.

\item \label{unique} If $\dim [I_Z]_i = \binom{i-t+2}{2}$ then it is clear that there is a unique form (up to scalar multiplication) in the initial degree $t$, and it is a common factor in degree $i$. We have to show that this is the only way a common factor can occur.

\item \label{33} Assume that we are not in the situation of item \ref{unique}, but that the elements of $[I_Z]_i$ have a common factor, $G$. Because of the genericity of the components of $Z$, we must have uniformity: $G$ must cut out a subscheme of the same degree on all the components of $Z$. If this multiplicity is 3, then in fact $G \in I_Z$ and we are in the situation of  item \ref{unique}. Thus, without loss of generality, we have one of the following:
\vspace{.1in}

\begin{enumerate}
    \item $G \in I_{Z_{s,1}}$ but not $I_{Z_{s,2}}$, and $I_Z : G =I_{Z_{s,2}}$, and so any $F \in [I_Z]_i$ satisfies $F = GG'$ with $G' \in I_{Z_{s,2}}$;

    \vspace{.1in}

    \item $G \in I_{Z_{s,2}}$ but not in $I_Z$, and $I_Z : G = I_{Z_{s,1}}$, and so any $F \in [I_Z]_i$ satisfies $F = GG'$ with $G' \in I_{Z_{s,1}}$.
\end{enumerate}

\noindent Either way, each $F \in [I_Z]_i$ is the product of an element of $I_{Z_{s,1}}$ and an element of $I_{Z_{s,2}}$.

\item If we show that $t \leq a+b-1$, we have a contradiction, since then the elements of $[I_Z]_t$ cannot be a product of the form described in item \ref{33}. By the definition of $a, b, t$,  it is enough to show that 
\[
\binom{a+b-1 +2}{2} \ge  \binom{a+2}{2} + \binom{b+2}{2}.
\]
This is equivalent to the assertion
\[
ab - a - b - 2 \geq 0
\]
or $(a-1) b \ge a+2$.
Since $a \ge 2$ and $b \ge 3$, this is true except in the case $a=2, b=3$. Thus, by definition of $b$, we get that $s$ is either 3 or 4. However, $\binom{6}{2} = 15 > 3 \cdot 4$ implies that $t \leq 4 < a+b$ in both cases, and so again we have our contradiction.

\end{enumerate}

\noindent This completes the proof of the claim.

The argument we have just given to show that there is no common factor can be modified to show that the elements of $[I_Z]_t$ are not composite with a pencil, since any element of the pencil would have to be an element of $I_{Z_{s,1}}$ at least.

By Bertini's theorem, we conclude that the general element of $[I_Z]_t$ is irreducible except in the situation described in the claim above. We now prove \Cref{GenFlatFat}. 

We again proceed by induction. First let $s=3$. It is clear that $Z_{3,3}$ lies on a cubic curve, namely a union of three lines. We want to show that this cubic is unique. We will produce a specific union of three flat fat points of multiplicity 3 that do not lie on a pencil of cubics, so then the result will follow by semicontinuity. 

Let $C$ be a smooth cubic curve and let $P_1, P_2$ be two flex points of $C$. Let $P_3$ be a general point on $C$. The divisor $Z = 3P_1 + 3P_2 + 2P_3$  on $C$ has degree 8, and so it lies on a pencil of cubics, which contains $C$. Since $P_3$ is not a flex point of $C$, vanishing triply at $P_3$ imposes an additional condition on the pencil. Hence,  $3P_1 + 3P_2 + 3P_3$ lies on a unique cubic, namely the union of three lines. 
Alternatively, one checks, for example by using Macaulay2, that the ideal 
\[
((x+y)^3, z) \cap ((x+z)^3, y) \cap ((y-z)^3, x)
\]
is generated in degree 3 by $xyz$. 

Assume that $s \ge 3$ and $Z_{s,3}$ has generic Hilbert function. We will add a new flat fat point of multiplicity 3 to produce $Z_{s+1,3}$. Let $\sigma$ be the initial degree of $I_{Z_{s,3}}$. If $\dim [I_{Z_{s,3}}]_\sigma = 1$ then imposing a new general point $P$ brings $\dim [I_{Z_{s,3}}]_\sigma = 0$. Now, if $Y =  Z_{s,3} \cup \{ P \}$, then $\dim [I_Y]_{\sigma+1} = \sigma+2$ (it has generic Hilbert function). We claim  that not all elements of the corresponding linear system are unions of lines. Indeed, if $s=3$, it is not hard to find a union of two conics that contains $Y$. If $s \geq 4$, then the condition $\dim [I_{Z_s,3}]_\sigma = 1$ and genericity forces $s \geq 9$. Then a degree consideration in fact forces {\it no} element of the corresponding linear system to be a union of lines. For example, when $s=9$ we have $\sigma = 6$; on the other hand, by genericity, none of the lines can contain more than a degree 3 subscheme of $Z_{9,3}$ (supported at a point) so we would need nine lines for the nonreduced components plus one more for $P$.

Thus we can find an element of $[I_Y]_{\sigma+1}$ containing a component that is not a line. (In fact, one can argue that the general element is irreducible, but we do not need this.) Let $Q$ be a general point of this component. By \Cref{rem:flex}, $Q$ is not a flex of this component. Thus, imposing tangency at this point imposes a second condition, and then imposing $m=3$ adds one final condition. Thus $Z_{s+1,3}$  has generic Hilbert function. 
\end{proof}

\begin{remark}
    We believe that the statement holds more generally. More precisely, for any $m\geq 4$, we expect that $Z_{s,m}$ has generic Hilbert function if and only if $s\geq 2m-3$.\newline
    It is not difficult to see that if $s<2m-3$, then $Z_{s,m}$ does not have generic Hilbert function. Indeed, in such a case, if $Z_{s,m}$ had generic Hilbert function, it would not be contained in a curve of degree $s$; however, it lies in the union of $s$ lines.
\end{remark}

%%%%%%%%%%%%%%%%%%%%%%%%%%%%%%%%%%%%%%%%%%%%%%%%%%%%%%%%%%%%%%%%%%%%%%%%%%%%%%%%%%%%%%%%%%%%%%%

%%%%%%%%%%%%%%%%%%%%%%%%%%%%%%%%%%%%%%%%%%%%%%%%%%%%%%%%%%%%%%%%%%%%%%%

\subsection{The case of double points}

\begin{proposition} \label{Z2}
Let $C \subset \PP^3$ be a set of $r$ general (disjoint) lines and let $H$ be a general hyperplane, defined by a linear form $L$. Let $Z_2$ be the zero-dimensional scheme defined by the saturation of the ideal $I_C + (L^2)$. Then $Z_2$ is in generic position on the surface $H^2$ defined by $L^2$, {i.e., 
\[
h_{Z_2} (j) = \min \{ 2r, (j+1)^2\}. 
\]
}
\end{proposition}

\begin{proof} 
Note that $Z_2$ has degree $2r$ by B\'ezout's theorem. We define the scheme $Z_1$ by
\[
I_{Z_1}  =  I_{Z_2} : (L) \subset R.
\]
Note that we also have 
\[
I_{Z_1 |H}  =  (I_{Z_2} + (L))^{sat}/(L) \subset R/(L) ,
\]
\[
 [I_{Z_2} + (L)]/(L) \cong I_{Z_2} / [I_{Z_2} \cap (L)] \cong  I_{Z_2} / [L \cdot [I_{Z_2} : (L) ]]
\]
and we have the exact sequence
\[
0 \rightarrow [I_{Z_2} : (L)](-1) \stackrel{\times L}{\longrightarrow} I_{Z_2} \rightarrow [I_{Z_2} + (L)]/(L) \rightarrow 0
\]
or (by sheafifying and taking cohomology)
\begin{equation} \label{usual exact seq}
0 \rightarrow I_{Z_1}(-1) \stackrel{\times L}{\longrightarrow} I_{Z_2} \rightarrow I_{Z_1 |H} \rightarrow H^1_*(\mathcal I_{Z_1})(-1) \rightarrow H^1_* (\mathcal I_{Z_2}) \rightarrow H^1_* (\mathcal I_{Z_1}) \rightarrow 0
\end{equation}
where $H^1_*$ refers to a direct sum over all twists. Here we have used the fact that if $X$ is a zero-dimensional scheme on a surface $F$ of degree $d$  in $\PP^3$ then we have the short exact sequence
\[
0 \rightarrow R(-d) \rightarrow I_X \rightarrow I_{X|F} \rightarrow 0
\]
so sheafifying and taking cohomology we have $H^1_*(\mathcal I_X) \cong H^1_*(\mathcal I_{X|F})$.

We argue by specialization. First, we give the general framework, then we give the specific details for our result.

1. Assume without loss of generality that $H$ is defined by $x_3$. Consider a component of $Z_2$ supported at a point $P$. This  component is defined by $I = (\ell_1, \ell_2, x_3^2)$. Then $I$ is determined by (i) the ideal of the point $P$, that is, $I_P = (\ell_1, \ell_2, x_3)$, and (ii) the line defined by $[I]_1$, the “tangent line,” because if $m$ is any linear form such that $(\ell_1, \ell_2, m) = (\ell_1, \ell_2, x_3)$ then 
$I = (\ell_1, \ell_2, m^2)$. 

2. Now we fix the reduced set $Z_1$ and we consider sets of $r$ lines, where for each point $P$ of $Z_1$ a line in $\mathbb P^3$ is chosen through $P$. Each such choice determines an ideal with local components as in 1. 
 Since $Z_2$ is determined by this information, it is specified by $Z_1$ plus a line through each point of $Z_1$.

We can take as our parameter space, then, the variety
\[
\underbrace{(\mathbb P^2) \times \dots \times  (\mathbb P^2)}_{\hbox{$r$ copies}}.
\]
If the line {$\lambda$ } through $P$ is general (we use $L$ to denote the linear form defining $H$), then the component of $Z_2$ supported at $P$ is defined by $I_\lambda + x_3^2$.
Thus, our $Z_2$ is a general set in the parameter space. 
However, if the line $\lambda$ is contained in $H$, i.e. $I_\lambda = (m_1, x_3)$, then the local component at $P$ is defined by 
$(m_1, x_3, m_2^2)$ with some $m_2$ chosen so that $I_P = (m_1, x_3, m_2)$. 

Our specialization, then, will come from moving general lines onto the plane $H$.

Let $t$ be the smallest integer such that $r \leq \binom{t+1}{2}$. Let $\alpha = \binom{t+1}{2} - r$. 
This means that the $h$-vector of $Z_1$ is
\[
(1,2,3,\dots, t-2, t-1, t-\alpha)
\]
where the $t-\alpha$ occurs in degree $t-1$. 
Note that since $Z_1$ can be taken to be a general set of $r$ points in $H$, $\dim [I_{Z_1 |H}]_{t-1} = \alpha$ and $t-1$ is the socle degree of the Artinian reduction of $\mathbb{K}[x,y,z]/I_{Z_1|H}$. If $r < \binom{t+1}{2}$, we also have that the initial degree of $I_{Z_1|H}$ is $t-1$; otherwise it is $t$.

For now assume that $r < \binom{t+1}{2}$.

Now we specialize. Choose any $\alpha$ of the nonreduced points of $Z_2$ and move the ``direction" as above so that it is a general direction on $H$.

Let $Y_2$ be the resulting scheme. Note $Y_2$ still lies in the surface $H^2$. 
We now define two related schemes in $H$:

\begin{itemize}
    \item First, we define $X_1$ by 
    $I_{X_1} = ( I_{Y_2} + (L))^{sat}$. We also have
    \[
    I_{X_1|H} = [[I_{Y_2} + (L)]/(L)]^{sat}.
    \]
It defines the union of  $r-\alpha$ general points of $H$ and $\alpha$ general nonreduced, length 2 schemes on $H$, for a total degree of $r+\alpha = \binom{t+1}{2}$.
    By Lemma \ref{GenFlatFat}, generic multiplicity~$2$ flat fat points impose independent conditions in $H$.  Hence the $h$-vector of $X_1$ is
    \[
(1,2,3,\dots, t-2, t-1, t).
\]

    \item Second, the scheme $X_2$ defined by $I_{X_2} = I_{Y_2} : (L)$ consists of (hence has the Hilbert function of) $r-\alpha$ general points in  $\PP^2 = H$.

\end{itemize}

\noindent The bottom line is that compared to the original scheme $Z_2$, here the zero-dimensional scheme on $H$ rises to degree $r+\alpha = \binom{t+1}{2}$ (the degree increases by $\alpha$), while the residual scheme drops in degree by $\alpha$. 

We will consider two cases, which we list first and then make the computations separately.

\vspace{.2in}

\noindent \underline{Case 1}: $\alpha \leq \frac{t}{2}$.

In this case the $h$-vectors of the ``residual" scheme defined by $ I_{X_2} = I_{Y_2} : (L)$, and that of the scheme defined by $I_{X_1} = [I_{Y_2} +(L)]^{sat}$ are
\[
(1,2,3,\dots, t-1,t-2\alpha) \ \ \hbox{ and } \ \ (1,2,3,\dots, t-1, t)
\]
where the last entries are both in degrees $t-1$.

\vspace{.2in}

\noindent \underline{Case 2}: $\alpha > \frac{t}{2}$.

In this case, for the residual scheme we have to reduce the total degree by $\alpha$ (compared to $Z_1$). Since $\alpha = (t-\alpha) + (2\alpha - t)$, we remove from the $h$-vector of $Z_1$ the $t-\alpha$ in degree $t-1$ plus an additional $(2\alpha-t)$ in degree $t-2$ to obtain the $h$-vector 
\[
(1,2,3, \dots, t-2, (t-1) - (2\alpha -t) ) = (1,2,3,\dots, t-2, 2t-2\alpha -1)
\]
for $X_2$, and we still have 
\[
(1,2,3,\dots, t)
\]
for $X_1$.
Here the last nonzero entries are in degrees $t-2$ and $t-1$ respectively. 

 For any integer $s$, we have the exact sequence
\begin{equation} \label{double2}
0 \rightarrow [I_{X_2}]_{s-1} \rightarrow [I_{Y_2 }]_s \rightarrow [I_{X_1 |H}]_s \rightarrow H^1(\mathcal I_{X_2}(s-1)) \rightarrow H^1 (\mathcal I_{Y_2}(s)) \rightarrow H^1 (\mathcal I_{X_1}(s)) \rightarrow 0.
\end{equation}

\noindent We observe that

\begin{itemize}

\item[(i)] $\dim [I_{X_1|H}]_s = 0$ for $s \leq t-1$.

\item[(ii)] $h^1(\mathcal I_{X_1}(s)) = 0$ for $s \geq t-1$.

\item[(iii)] In Case 1, $h^1(\mathcal I_{X_2} (s-1)) = 0$ for $s \geq t$.

\item[(iv)] In Case 2, $h^1(\mathcal I_{X_2} (s-1)) = 0$ for $s \geq t-1$.
\end{itemize}

We begin with Case 1. From exactness of (\ref{double2}) and (ii), (iii) we conclude

\begin{enumerate}

\item 
\[
\left.
\begin{array}{l}
h^1(\mathcal I_{X_2}(t-1)) = 0 \\
h^1(\mathcal I_{X_1}(t)) = h^1(\mathcal I_{X_1}(t-1)) = 0
\end{array}
\right \}
\Rightarrow
h^1(\mathcal I_{Y_2}(s)) = 0 \hbox{ for } s \geq t.
\]

\item $\dim [I_{Y_2}]_s = \dim [I_{X_2}]_{s-1}$ for $0 \leq s \leq t-1$, from (i).
    
\end{enumerate}

Since the $h$-vector of $X_2$ is
\[
(1,2,3,\dots,t-1,t-2\alpha)
\]
but now we consider $X_2$ as being in $\mathbb P^3$, we have for $s \leq t-1$ that 
\begin{equation} \label{dimY2}
\dim [I_{Y_2}]_s = \dim [I_{X_2}]_{s-1} =  \binom{s+1}{3} 
\end{equation}
(the ideal coincides with the ideal of the plane). 
Thus for $s \leq t-1$,
\[
h_{Y_2}(s) = \binom{s+3}{3} - \dim [I_{Y_2}]_s = (s+1)^2.
\]
This means that the $h$-vector of ${Y_2}$ has the form
\[
(1,3,5,7,\dots)
\]
until degree $t-1$. On the other hand, conclusion 1. above shows that the $h$-vector of $Y_2$ ends in degree at most $t$. Thus $Y_2$ has generic Hilbert function in the surface $H^2$.

We now turn to Case 2, {where $\alpha > \frac{t}{2}$.} Now we have (again from the exactness of (\ref{double2}))

\begin{enumerate}

    \item 
    \[
    \left. 
\begin{array}{l}
h^1(\mathcal I_{X_2}(t-2)) = 0 \Rightarrow h^1(\mathcal I_{X_2}(t-1)) = 0 \\
h^1(\mathcal I_{X_1}(t)) = h^1(\mathcal I_{X_1}(t-1)) = 0
\end{array}
\right \}
\Rightarrow h^1(\mathcal I_{Y_2}(s)) = 0 \hbox{ for } s \geq t
    \]
    (using (ii) and (iv)).

    \item $\dim [I_{Y_2}]_s = \dim [I_{X_2}]_{s-1}$ for $0 \leq s \leq t-1$.
    
\end{enumerate}

This information is identical to that of Case 1, so the argument is also identical and (\ref{dimY2}) again holds.

Now, in both cases $Y_2$ is a specialization of $Z_2$, so for any $s$ we have by semicontinuity
\[
\dim [I_{Y_2}]_s \geq \dim [I_{Z_2}]_s .
\]
But both lie on the surface $H^2$, whose ideal has dimension $\binom{s+1}{3}$ in degree $s$. We have, using (\ref{dimY2}) for $s \leq t-1$, 
\[
\binom{s+1}{3} = \dim [I_{Y_2}]_s \geq \dim [I_{Z_2}]_s \geq \binom{s+1}{3}.
\]
Thus the original scheme $Z_2$ also has generic Hilbert function in $H^2$, since the $h$-vector ends in degree $t$.

Finally, when we introduced our specialization, we left open the case where $r = \binom{t+1}{2}$. In this case, no specialization is needed! $X_1$ already has $h$-vector $(1,2,\dots,t)$ and the argument goes as above.
\end{proof}

As an immediate corollary of what we have shown, we state:

\begin{corollary} \label{cor: genHF}
    
    The scheme $Z_2$ has the  $h$-vector 
    \[
    (1,3,5,7,\dots, 2t-1,\alpha)
    \]
    where $1 \leq \alpha \leq 2t+1$. In particular, we have that $H^0(\mathcal I_{Z_2 |H^2}(k))$ and $H^1(\mathcal I_{Z_2}(k))$ can never both be nonzero, for any $k$.  
\end{corollary}

\begin{proof}
    It is an immediate consequence of Proposition \ref{Z2}.
    The last sentence comes from the fact that 
    \[
   \min \{ t | H^1(\mathcal I_{Z_2}(t)) = 0 \} =  \min \{ t | H^1(\mathcal I_{Z_2|H^2}(t)) = 0 \} 
\]
comes at the end of the $h$-vector, which is where $I_{Z_2}$ begins to have elements that are not in $I_{H^2}$.
\end{proof}

\begin{corollary}\label{SLP2}
    If $C \subset \mathbb P^3$ is a general set of skew lines then for a general linear form $L$ and any integer $s$,
    \[
    \times L^2 \colon [M(C)]_{s-2} \rightarrow  [M(C)]_s
    \]
     has maximal rank.
\end{corollary}

\begin{proof}
  It follows immediately from the long exact sequence
  \[
  0 \rightarrow [I_C]_{s-2} \stackrel{\times L^2}{\longrightarrow} [I_C]_s \rightarrow [I_{Z_2 | H^2}]_s \rightarrow [M(C)]_{s-2} \stackrel{\times L^2}{\longrightarrow} [M(C)]_s \rightarrow H^1(\mathcal I_{Z_2} (s)) \rightarrow \dots
  \]
  and Corollary \ref{cor: genHF}.
\end{proof}

%%%%%%%%%%%%%%%%%%%%%%%%%%%%%%%%%%%%%%%%%%%%%%%%%%%%%%%%%%%%%%%%%%%%%%%%%%%%%%%%%%%%%%%%%

\subsection{The case of flat fat points of multiplicity three}

The goal of this subsection is to prove the following theorem, and then to show that if $C$ is a set of general lines then $M(C)$ has SLP in range 3.

\begin{theorem} \label{Z3}
Let $C \subset \PP^3$ be a set of $r$ general (disjoint) lines and let $H$ be a general hyperplane, defined by a linear form $L$. Let $Z_3 \subset \PP^3$ be the zero-dimensional scheme defined by the saturation of the ideal $I_C + (L^3)$. Then $Z_3$ is in generic position on the surface $H^3$ defined by $L^3$, {i.e. 
\[
h_{Z_3} (j) = \min \left\{3 r, 3 \binom{j+1}{2} +1 \right\}. 
\]
} 
\end{theorem}

We prove Theorem \ref{Z3} in this subsection. The strategy will be to find a different scheme $X_3$ that is a specialization of $Z_3$ and also has generic Hilbert function on a cubic surface. This means that the ideal of $X_3$ has the smallest possible dimension in each degree, so the ideal for $Z_3$ must also be minimal.

Note that $Z_3$ has degree $3r$ by B\'ezout's theorem. We define the scheme $Z_2$ by
\[
I_{Z_2}  =  I_{Z_3} : (L) \subset R
\]
and the scheme $Z_1 \subset \PP^3$ by $I_{Z_1} = (I_C + (L))^{sat}$. 
Note that we also have 
\[
I_{Z_1 |H}  =  (I_{Z_3} + (L))^{sat}/(L) \subset R/(L) ,
\]
\[
 [I_{Z_3} + (L)]/(L) \cong I_{Z_3} / [I_{Z_3} \cap (L)] \cong  I_{Z_3} / [L \cdot [I_{Z_3} : (L) ]]
\]
and we have the exact sequence
\[
0 \rightarrow [I_{Z_3} : (L)](-1) \stackrel{\times L}{\longrightarrow} I_{Z_3} \rightarrow [I_{Z_3} + (L)]/(L) \rightarrow 0
\]
or (by sheafifying and taking cohomology)
\begin{equation}\label{exact seq fat3}
0 \rightarrow I_{Z_2}(-1) \stackrel{\times L}{\longrightarrow} I_{Z_3} \rightarrow I_{Z_1 |H} \rightarrow H^1_*(\mathcal I_{Z_2})(-1) \rightarrow H^1_* (\mathcal I_{Z_3}) \rightarrow H^1_* (\mathcal I_{Z_1}) \rightarrow 0
\end{equation}
where $H^1_*$ refers to a direct sum over all twists. Here we have used the fact that if $X$ is a zero-dimensional scheme on a surface $F$ of degree $d$  in $\PP^3$ then we have the short exact sequence
\[
0 \rightarrow R(-d) \rightarrow I_X \rightarrow I_{X|F} \rightarrow 0
\]
so sheafifying and taking cohomology we have $H^1_*(\mathcal I_X) \cong H^1_*(\mathcal I_{X|F})$.

We argue by specialization. First we give the general framework, then we provide the specific details for our result.

1. Assume without loss of generality that $H$ is defined by $x_3$. Consider a component of $Z_3$ supported at a point $P$. This  component is defined by $I = (\ell_1, \ell_2, x_3^3)$. Then $I$ is determined by (i) the ideal of the point $P$, that is, $I_P = (\ell_1, \ell_2, x_3)$,  (ii) the line defined by $[I]_1$, the “tangent line,” and (iii) the fact that the multiplicity is 3 (a third condition as noted above).

2. Now we fix the reduced set $Z_1$ and we consider sets of $r$ lines, where for each point $P$ of $Z_1$ a line in $\mathbb P^3$ is chosen through $P$. Each such choice determines an ideal with local components as in 1. 
 Since $Z_3$ is determined by this information, it is specified by $Z_1$ plus a line through each point of $Z_1$.

We can take as our parameter space, then, the variety
\[
\underbrace{(\mathbb P^2) \times \dots \times  (\mathbb P^2)}_{ \hbox{$r$ copies}}.
\]
If the line $\lambda$ through $P$ is general, the local component at $P$ is defined by $I_\lambda + x_3^3$.
Thus, our $Z_3$ is a general set in the parameter space. 
However, if the line $\lambda$ is contained in $H$, i.e., $I_\lambda = (m_1, x_3)$, then the local component at $P$ is defined by 
$(m_1, x_3, m_2^3)$ with $m_2$ chosen so that $I_P = (m_1, x_3, m_2)$. 

Our specialization, then, will come from moving general lines onto the plane $H$. We will also add general reduced points to this specialization, as described next. Before doing so, we give an example to illustrate the need for specialization.

\begin{example} \label{useful ex}
    Let $C$ be a set of 29 general lines in $\mathbb P^3$ and let $Z_1, Z_2, Z_3$ be as defined above. By what we have seen above, we know the Hilbert functions of $Z_1$ and $Z_2$. Their $h$-vectors are $(1,2,3,4,5,6,7,1)$ and $(1,3,5,7,9,11,13,9)$ respectively.  From these we can compute the Hilbert function and hence the dimensions of the components of the ideals.
    
    The exact sequence (\ref{exact seq fat3}) gives the following table of dimensions for each degree $t$. We only give the values that are important for our calculation.
{\scriptsize
\[
\begin{array}{ccccccccccccccccc}
0 & \rightarrow & [I_{Z_2}]_{t-1} & \stackrel{\times L}{\longrightarrow} & [I_{Z_3}]_t & \rightarrow & [I_{Z_1 |H}]_t & \rightarrow & H^1(\mathcal I_{Z_2}(t-1))  & \rightarrow & H^1 (\mathcal I_{Z_3}(t)) & \rightarrow & H^1 (\mathcal I_{Z_1}(t)) & \rightarrow \\
t=0 && 0 &&&& 0   \\
t=1 && 0 &&&& 0  \\
t=2 && 0 &&&& 0 \\
t=3 && 1 &&&& 0 \\
t=4 && 4 &&&& 0 \\
t=5 && 10 &&&& 0 \\
t=6 && 20 &&&& 0 \\ 
t=7 && 35 &&&& 7 && 9 &&&& 0 \\
t=8 && 62 && && 16 && 0
\end{array}
\]
}
As a result, we can immediately obtain the value of the Hilbert function of $Z_3$ in degrees $\leq 6$ and degree 8. Specifically, we have
\[
h_{Z_3} = (1,4, 10, 19, 31, 46, 64, ??, 87,87, \ldots)
\]
or $h$-vector $(1,3,6,9,12,15,18,a,b)$ where $a+b=23$. 

The central issue is that in degree $\ 7$, exactness allows for different values of $\dim [I_{Z_3}]_7$ and $h^1(\mathcal I_{Z_3}(7))$. We seek a specialization $X_3$ for which, in each degree $t$, either $\dim [I_{{X_1 |H}}]_t = 0$ or $h^1(\mathcal I_{X_2} (t-1)) = 0$;  this guarantees that we can compute $\dim [I_{X_3}]_t$ in the specialization, and if it is the general value that we seek, then $I_{Z_3}$ will also have the desired Hilbert function.  
\end{example}

We also have the following relations between the $h$-vectors of $Z_1$ and of $Z_2$ in general. Recall that the socle degree is the last degree in which the $h$-vector is nonzero.

\begin{lemma} \label{compare soc deg}
    Assume that $Z_1$ has socle degree $t$. Write $(h_i)$ for the $h$-vector of $Z_1$ and $(k_i)$ for the $h$-vector of $Z_2$. Then 
    
    \begin{itemize}
        \item[(i)] $Z_2$ has socle degree either $t$ or $t+1$.

        \item[(ii)] If the socle degree of $Z_2$ is $t$ then 
        \[
        k_t \geq t+1 \ \ \hbox{ and } \ \ 
2h_t +t = k_t.
\]

\item[(iii)] If the socle degree of $Z_2$ is $t+1$ then
\[
k_{t+1} = 2h_t -t-1 .
\]
    \end{itemize}
\end{lemma}

\begin{proof}

    The choice of $t$ means that since $Z_1$ has $h$-vector $(1,2,3,\dots,t,b)$ for some $b>0$ and $r = \deg Z_1$,
    \[
   1 + 2 + \dots + t = \frac{t^2+t}{2} = \binom{t+1}{2} < r \leq \binom{t+2}{2} = \frac{t^2+3t+2}{2}.
    \]
    This means that 
    \begin{equation} \label{ineqs}
    t^2 < t^2+t < 2r = \deg Z_2 \leq t^2+3t+2 = (t+1)^2 +t+1 < (t+2)^2.
    \end{equation}
    The $h$-vector of $Z_2$ has the form $(1,3,5,7,\dots, 2t+1,\dots)$ (Proposition \ref{Z2}) and the sum of the entries up to degree $t$ is $(t+1)^2$. The first strict inequality in (\ref{ineqs}) means that the socle degree of $Z_2$ is at least $t$, and the last inequality shows that the socle degree is at most $t+1$. This proves (i). 

    For (ii), note from above
    \[
  1+3+5+ \dots + (2t-1) =  t^2 < t^2+t < 2r = \deg Z_2 \leq (t+2)^2, 
    \]
    and so the $h$-vector of $Z_2$ has the form $(1,3,5,\dots,2t-1,k_t)$, where $k_t$ satisfies 
 $t+1 \leq k_t \leq 2t+1$ {as $\deg Z_2 = t^2 + k_t > t^2 + t$ by Inequality \eqref{ineqs}.}
Then we have
\[
2 \cdot \deg Z_1 = 2 \binom{t+1}{2} + 2h_t = \deg Z_2 = t^2 + k_t, 
\]
from which the second part of (ii) follows.

Part (iii) follows by a similar argument.
\end{proof}

We now describe our specialization. As in the previous subsection, it consists of moving a predetermined number of components of $Z_3$ to the plane $H$. In a moment, we will specify how many such components will be moved, but first we make some general observations.

\renewcommand{\theenumi}{\roman{enumi}}
\renewcommand{\labelenumi}{(\theenumi)}

\begin{enumerate}
    \item Each such component of $Z_3$ is a flat fat point of multiplicity 3 on the corresponding line.

    \item We specialize certain components to $H$ so that the new direction is general in $H$. By 
Lemma \ref{GenFlatFat}, the resulting union of flat fat points in $H$ has generic Hilbert function. 

\item The other lines meet $H$ in general points of $H$, so the union of these (reduced) points with the new flat fat points in $H$ still has generic Hilbert function. Let $X_1$ be this (general) union of simple points and multiplicity 3 flat fat points.

\item \label{obs4} Let $X_3$ be the new union of flat fat point schemes {obtained from $Z_3$}, after moving a certain number of components to $H$. Let $k$ be the number of moved components, so $r-k$ components have not moved. Let $X_2$ be defined by $I_{X_2} = I_{X_3} : L$, where $L$ is the linear form defining $H$. Then $X_2$ is the union of $r-k$ flat fat points in $\mathbb P^3$ of multiplicity 2 -- the components that were moved disappear from $X_2$ since they lie entirely on $H$.

\item $I_{X_1|H} = \frac{{(I_{X_3} + (L))^{sat}}}{(L)}. $

\item Each time that we specialize a component {of $Z_3$}, the degree of $X_1$ increases by 2 and the degree of $X_2$ decreases by 2. The degree of $X_3$ remains unchanged.
\end{enumerate}

Now we describe the process of deciding how many components we want to specialize to $H$. Assume that the socle degree of $Z_1$ is $t$, as in Lemma \ref{compare soc deg}, and the $h$-vector of $Z_1$ in degree $t$ is $h_t$. Note $h_t \leq t+1$. There are two cases.

\vspace{.2in}

\noindent {\bf (I.)}  If the socle degree of $Z_2$ is also $t$, then the $h$-vector of $Z_2$ in degree $t$ is $k_t \geq t+1$ and $k_t = 2h_t + t$ (by Lemma \ref{compare soc deg}). We specialize enough points so that $h_t$ rises to $t+1$ (i.e. $[I_{X_1}]_t = 0$). To do this, we have to fill a gap of $(t+1) - h_t$ in steps of 2, and so in the worst case the degree from $Z_1$ to $X_1$ goes up by $(t+2) - h_t$. But note that
\[
k_t - (t+2-h_t) = 2h_t + t - (t+2-h_t) = 3h_t-2 > 0. 
\]
Thus, we are assured that after specializing enough components so that $h_t$ rises to $t+1$, the socle degree of $X_2$ remains $t$.

\vspace{.2in}

\noindent {\bf (II.)} If the socle degree of $Z_2$ is $t+1$,  the $h$-vector of  $Z_2$ (before any specialization) is of the form 
\[
(1,3,5,7,9,\dots,2t+1, k_{t+1}).
\]
{We proceed in two steps.} 
For step 1, we specialize enough lines so that the $h$-vector of $Z_2$ goes to zero in degree $t+1$, i.e. we specialize $\lceil \frac{k_{t+1}}{2} \rceil$ points of $Z_2$. The degree of $Z_1$ has gone up to get $X_1$ by the same amount that $Z_2$ went down to get $X_2$. There are two possibilities.

\vspace{.2in}

\noindent {\bf (II a.)} If the $h$-vector of $X_1$ has $h_t = t+1$, then we stop.

\vspace{.2in}

\noindent {\bf (II b.)} It may happen that the $h$-vector of $X_1$ in degree $t$ is less than $t+1$. We first illustrate this with an example. 

 \begin{example}
     Let $C$ be a general set of 25 lines in $\mathbb P^3$. The $h$-vector of $Z_1$ is
     \[
     (1,2,3,4,5,6,4)
     \]
     and that of $Z_2$ is 
     \[
     (1,3,5,7,9,11,13,1).
     \]
     Then for step 1 we only specialize one component, giving $X_1$ and $X_2$ with $h$-vectors
     \[
     (1,2,3,4,5,6,6) \ \ \hbox{ and } \ \ (1,3,5,7,9,11,12).
     \]
As we saw in Example \ref{useful ex} looking at the line where $t=7$, we want to avoid a situation where $\dim [I_{X_1}]_t$ and $h^1(\mathcal I_{Z_2}(t-1))$ are both nonzero. Here it fails when $t=6$. The solution, step 2, is to specialize one more component. The new $h$-vectors are 
     \[
     (1,2,3,4,5,6,7,1) \ \ \hbox{ and } \ \ (1,3,5,7,9,11,10), 
     \]
and we have avoided the stated situation.
 \end{example}

Thus, step 2 in general is to specialize additional components so that the $h$-vector of $X_1$ has the value $t+1$ in degree $t$. Note that, in making these specializations, $h_t$ must rise by at most $t$ (from 1 to $t+1$). Moreover, under our assumption that the socle degree of $Z_2$ is $t+1$,  the socle degree of $X_2$ is at least $t$, since $k_{t+1} + 2t + 1 > 2 ( \lceil \frac{k_{t+1}+1}{2} \rceil + \lceil \frac{t+1 - h_t}{2} \rceil)$.
Thus, after these specializations, we have 
\begin{itemize}
    \item $h^1(\mathcal I_{X_2}(t)) = 0$;
    \item $h^1(\mathcal I_{X_1}(t+1)) = 0$;
    \item $\dim [I_{X_1|H}]_t = 0$.
\end{itemize}  

At this point we can make some conclusions about  $X_3$ using the long exact sequence (\ref{exact seq fat3}).

\renewcommand{\theenumi}{\arabic{enumi}}
\renewcommand{\labelenumi}{\theenumi.}

\begin{enumerate}
    \item For every degree $s \leq t$, we have 
    \[
    \dim [I_{X_3}]_s = \dim [I_{X_2}]_{s-1} = \dim [R]_{s-3}
    \]
    since $I_{X_2}$ is equal to the principal ideal $(L^2)$ in degrees $\leq t-1$ and we have a shift of~1. Thus $I_{X_3}$ agrees with the ideal $(L^3)$ in degrees $\leq t$.

    \item $h^1(\mathcal I_{X_1}(t+1)) = h^1(\mathcal I_{X_2}(t)) = 0$ so $h^1(\mathcal I_{X_3}(t+1)) = 0$ by exactness.
\end{enumerate}

\noindent These two facts imply that $X_3$ has generic Hilbert function on a surface of degree 3.

The last step is to show that $Z_3$ also has generic Hilbert function on a surface of degree~3. But the ideal $I_{X_3}$ has the smallest possible dimension in each degree. Since $X_3$ is a specialization of $Z_3$, the same is true of  $Z_3$. This concludes the proof of Theorem \ref{Z3}.
\smallskip

As an immediate corollary of what we have shown, we state:

\begin{corollary} \label{cor: genHF3}
    
    The scheme $Z_3$ has the  $h$-vector 
    \[
    (1,3,6,9,12,\dots, 3t,\alpha)
    \]
    where $1 \leq \alpha \leq 3(t+1)$. In particular, $H^0(\mathcal I_{Z_3 |H^3}(k))$ and $H^1(\mathcal I_{Z_3}(k))$ can never both be nonzero, for any $k$.  
\end{corollary}

\begin{proof}

    It is an immediate consequence of Theorem \ref{Z3}.
    The last sentence comes from the fact that 
    \[
   \min \{ t \mid H^1(\mathcal I_{Z_3}(t)) = 0 \} =  \min \{ t \mid H^1(\mathcal I_{Z_3|H^3}(t)) = 0 \} 
\]
comes at the end of the $h$-vector, which is where $I_{Z_3}$ begins to have elements that are not in $I_{H^3}$.
\end{proof}

\begin{corollary} \label{SLP3}
    If $C \subset \mathbb P^3$ is a general set of $r$ skew lines then for a general linear form $L$ and any integer $s$,
    \[
    \times L^3 \colon [M(C)]_{s-3} \rightarrow  [M(C)]_s
    \]
     has maximal rank.
\end{corollary}

\begin{proof}
  It follows immediately from the long exact sequence
  \[
  0 \rightarrow [I_C]_{s-3} \stackrel{\times L^3}{\longrightarrow} [I_C]_s \rightarrow [I_{Z_3 | H^3}]_s \rightarrow [M(C)]_{s-3} \stackrel{\times L^3}{\longrightarrow} [M(C)]_s \rightarrow H^1(\mathcal I_{Z_3} (s)) \rightarrow \dots
  \]
  and Corollary \ref{cor: genHF3}.
\end{proof}

We believe that the pattern shown in this section continues:

\begin{conjecture}
    If $C$ is a general set of $r$ skew lines in $\mathbb P^3$ then $M(C)$ has SLP.
\end{conjecture}

%%%%%%%%%%%%%%%%%%%%%%%%%%%%%%%%%%%%%%%%%%%%%%%%%%%%%%%%%%%%%%%%%%%%%%%

\section{Skew lines on a quadric surface, and related sets}
\label{sec4}

In this section, we explore Lefschetz properties of unions of skew lines in $\PP^3$ when all or most of the lines lie on a quadric surface.

\subsection{Unions of lines on a smooth quadric surface}

Let $C$ be a set of skew lines on a quadric surface $Q$. We first give some results about $M(C)$, and then make slight changes to $C$ (allowing one or two lines to be off the quadric) and see how this changes the behavior with respect to WLP.

\begin{lemma} 
The Hilbert function of $M(C)$ is symmetric.
\end{lemma}

\begin{proof}
    This is clear since $C$ lies entirely in one ruling of $Q$, and is directly linked to a set of the same number of skew lines in the other ruling, so the dimensions of the components of $M(C)$ are symmetric thanks to the behavior of $M(C)$ under linkage (see \cite{MBook}).
\end{proof}

In fact, we have the free resolution of $M(C)$:

\begin{proposition}
    Let $C \subset \PP^3$ be a set of $r$ skew lines on a smooth quadric surface. Then $M(C)$ has a minimal free resolution of the form
    \[
0 \rightarrow R(-r-2)^{r-1} \rightarrow R(-r-1)^{2r} \rightarrow 
\begin{array}{c}
R(-2)^{r+1} \\
\oplus \\
R(-r)^{r+1}
\end{array}
\rightarrow R(-1)^{2r} \rightarrow R^{r-1} \rightarrow M(C) \rightarrow 0.
    \]
    In particular, the resolution is symmetric,  as is the Hilbert function of $M(C)$ itself. The socle is concentrated in the last degree, and $M(C)$ is generated in its initial degree.
\end{proposition}

\begin{proof}
    From the exact sequence (\ref{usual 1}) we get a short exact sequence
    \[
    0 \rightarrow R/I_C \rightarrow \bigoplus^r H^0_*(\mathcal O_{\PP^1}) \rightarrow M(C) \rightarrow 0
    \]
    where $H^0_*$ refers to a direct sum over all twists. Note that $\bigoplus^r H^0_*(\mathcal O_{\PP^1})$ is a direct sum of copies of $\mathbb K[x,y]$ or $R/I_\lambda$, where $\lambda$ is a line in $\PP^3$.

    Because $C$ is linked by a complete intersection of type $(2,r)$ to a set of $r$ skew lines in the other ruling of the quadric $Q$, we compute that $M(C)$ begins in degree 0 and ends in degree $r-2$ (using the standard relation between the Hartshorne-Rao modules of linked schemes -- see, for instance, Theorem 5.3.1 of \cite{MBook}). Also, the above exact sequence shows that $M(C)$ is generated in its initial degree and $\dim [M(C)]_0 = r-1$. Finally, although we do not claim that $M(C)$ is self-dual up to shift, it is true that the dual module (shifted) is the Hartshorne-Rao module of a set of $r$ skew lines on $Q$ (in the other ruling), so the Betti numbers will be symmetric and we only have to compute the first half.

We begin by writing what we know about minimal free resolutions of the modules in the above exact sequence:
\[
\begin{array}{cccccccccccc}
&& 0 \\
&& \downarrow \\
&& A && 0 \\
&& \downarrow && \downarrow \\
&& B && R(-2)^r \\
&& \downarrow && \downarrow \\
&& R(-2) \oplus R(-r)^a && R(-1)^{2r} \\
&& \downarrow && \downarrow \\
&& R && R^r \\
&& \downarrow && \downarrow \\
0 & \rightarrow & R/I_C & \rightarrow & \bigoplus^r (R/I_\lambda) & \rightarrow & M(C) & \rightarrow & 0 \\
&& \downarrow && \downarrow \\
&& 0 && 0
\end{array}
\]
where $A$ and $B$ are free modules and $a$ is a positive integer. We suppress the horizontal arrows. We know that $M(C)$ starts in degree 0 and has dimension $r-1$, so one copy of $R$ splits off. Then, by a mapping cone, $M(C)$ has free resolution
\[
0 \rightarrow A \rightarrow B \rightarrow 
\begin{array}{c}
R(-2)^{r+1} \\
\oplus \\
R(-r)^a
\end{array}
\rightarrow R(-1)^{2r} \rightarrow R^{r-1} \rightarrow M(C) \rightarrow 0.
\]
Now, the fact that the $\mathbb K$-dual of $M(C)$ is again the Hartshorne-Rao module of a set of $r$ skew lines  on $Q$ (beginning in degree 0 and ending in degree $r-2$) means that $A = R(-r-2)^{r-1}$ and $B = R(-r-1)^{2r}$. A rank calculation gives $a = r+1$. There is clearly no possible splitting, so the resolution is minimal and we are done.
\end{proof}

\begin{remark} \label{enough}
    If $C$ is a set of $r$ skew lines on a quadric surface, then the socle of $M(C)$ is supported in the last degree of $M(C)$ (namely $r-2$). It follows that checking for WLP can be done by checking only for injectivity, and only from degree $\lfloor \frac{r-2}{2} \rfloor -1$ to degree $\lfloor \frac{r-2}{2} \rfloor$. The surjectivity comes from the fact that the dual module also corresponds to $r$ skew lines on $Q$ (in the other ruling). Equivalently (and possibly more easily) it suffices to show surjectivity from degree $\lfloor \frac{r-2}{2} \rfloor$ to degree $\lfloor \frac{r-2}{2} \rfloor +1$, and then injectivity will follow. See, for instance, \cite[Proposition 2.6]{BMMN}  and the paragraph preceding it for the (similar) Gorenstein situation.
\end{remark}
\begin{proposition}\label{prop:lines quadric}
    Let $C$ be a set of $r$ skew lines on a smooth quadric surface. Then $M(C)$ has the WLP.
\end{proposition}
\begin{proof}
    Let $Q$ be the polynomial defining the quadric surface. Then for $t\leq \frac{r-2}{2}$ we have $[I_{C}]_t=QR_{t-2}$, with an analogous statement for $C \cap H$ in the plane $H$. Let $H$ be a general plane and $t\leq \frac{r-2}{2}$. It follows that in the exact sequence $$
    0\to [I_C]_{t-1}\to [I_C]_t\to [I_{C\cap H | H}]_t\to [M(C)]_{t-1}\to [M(C)]_t\to\dots
    $$
    the alternating sum of the dimensions of the first three terms is $\binom{t}{3}-\binom{t+1}{3}+\binom{t}{2}=0$. In particular, we have injectivity from degree $\lfloor \frac{r-2}{2}\rfloor-1$ to degree $\lfloor \frac{r-2}{2}\rfloor$ so we are done by Remark \ref{enough}. 
\end{proof}

\begin{corollary} \label{curve on quadric}
    If $C$ is any curve on a smooth quadric surface then $M(C)$ has WLP. 
\end{corollary}

\begin{proof}
    This follows by liaison. If $C$ has type $(a,a)$ then it is a complete intersection so $M(C) = 0$. If $C$ is a curve of type $(a,b)$ with $a < b$ then $C$ is directly linked to $b-a$ lines on the quadric.
\end{proof}

\subsection{All but one line on a quadric}

\begin{proposition}
\label{all but one}
    Let $C$ be the union of a set of $r$ skew lines on a smooth quadric surface $Q$ and a general line $\ell$. Then $M(C)$ has WLP.
\end{proposition}
\begin{proof}
    Note that for $t<r$, $[I_C]_t=Q[I_\ell]_{t-2}$. In particular $\dim [I_{C}]_t={t+1\choose 3}-(t-1)$. It follows from \eqref{usual 1} that $\dim [M(C)]_t=(t+1)(r+1)-{t+3\choose 3}+{t+1\choose 3}-(t-1)$. By direct computation, we have $\dim [M(C)]_t-\dim [M(C)]_{t-1}=r-2t-1$; thus $\dim [M(C)]_t$  is maximal at $t=\left\lfloor \frac{r-1}{2}\right\rfloor$.

    Let $H$ be a general hyperplane, so $(C\cap H)|H$ consists of $r$ points on a conic and a general point. Therefore $[I_{C\cap H|H}]_t=Q'[I_{\ell\cap H}]_{t-2}$ for $t\leq \frac{r-1}{2}$, where $Q'$ is the equation of $(Q\cap H)|H$. It follows that $\dim [I_{C\cap H|H}]_t=
        {t\choose2}-1$, if $ t\leq \frac{r-1}{2}$, and $\dim [I_{C\cap H|H}]_t=
        {t+2\choose 2}-r-1$ when $t>\frac{r-1}{2}.$
    
    For $t>\left\lfloor \frac{r-1}{2}\right\rfloor$, surjectivity follows from \eqref{usual 3} by noting that $h^1(\mathcal I_{C\cap H|H}(t))=0$.

    For $t\leq \left\lfloor \frac{r-1}{2}\right\rfloor$, injectivity follows by considering the alternating sum of the dimensions of the first three terms in \eqref{usual 3}, which is $$\left[{t\choose 3}-(t-2)\right]-\left[{t+1\choose 3}-(t-1)\right]+\left[{t\choose2}-1\right]=0.$$

\end{proof}

\subsection{All but two lines on a quadric}
\label{subsec: all but two}

\begin{proposition} \label{all but 2}
Let $C$ be a set of $s$ skew lines with $s-2$ of them on a quadric surface $Q$ and two skew lines not lying on $Q$. If $s \geq 10$ then $M(C)$ fails WLP from degree 2 to degree 3.
\end{proposition}

\begin{proof}

Note that for the union of $s-2 \geq 8$ lines  contained in $Q$ and two skew lines not lying on $Q$ we have $\dim [I_C]_t = 0$ for $t \leq 3$.  From the sequence (\ref{usual 1}) we get
\[
\dim [M(C)]_2 = 3s-10 \ \ \hbox{ and } \ \ \dim [M(C)]_3 = 4s-20,
\]
so for $s \geq 10$ we have $\dim [M(C)]_2 \leq \dim [M(C)]_3$. Then by (\ref{usual 3}) (with $m=1$), WLP fails from degree 2 to degree 3 if and only if $\dim [I_{C \cap H | H}]_3 \neq 0$.  But the latter inequality is clear since we have $s-2$ points on a conic and two off the conic.
\end{proof}

%%%%%%%%%%%%%%%%%%%%%%%%%%%%%%%%%%%%%%%%%%%%%%%%%%%%%%%%%%%%%%%%%%%%%%%

\section{Some comments on smooth irreducible curves}
\label{sec5}

Of course, sets of skew lines are smooth, and we have given some analysis of such curves in the earlier sections. A more delicate setting is that of smooth irreducible curves. We know that given any graded module $M$ of finite length, there is a smooth curve $C$ for which $M(C) \cong M(-\delta)$ for $\delta \gg 0$ \cite{rao}. {On the other hand, since $M$ has finite length, for $\delta \gg 0$ the component in degree 0 of $M(-\delta)$ is 0, so such a curve is connected. Thus $C$ is smooth and irreducible.}

However, such curves have very large degree in general.

\begin{question} \label{Q for curves}
What is the smallest degree, $d$, of a {\it smooth, irreducible} curve $C$ for which $M(C)$ fails WLP? What is the smallest arithmetic genus of such a curve?
\end{question}

While we do not have a complete answer, the results of this  section illustrate the strength of the Huneke-Ulrich result (Theorem \ref{HU fact}), and Remark \ref{use irred} shows how irreducibility can be a key factor.

First, the next result gives a candidate that we believe has the smallest degree and genus.  We noted above that Buchsbaum liaison classes are a natural place to look for curves for which $M(C)$ does not have WLP since the multiplication maps are zero. It is certainly true that the curve constructed in Proposition \ref{11buchs} below is the smallest smooth, irreducible Buchsbaum curve whose module has diameter at least two, as required for WLP failure to be relevant to our problem. We refer to \cite{BBM} and \cite{BM4}  for results about the Lazarsfeld-Rao Property and shifts of the even liaison class, and to \cite{BM2} and \cite{GM1} for facts about arithmetically Buchsbaum curves.

We recall some notation and facts. Let $M$ be any graded $R$-module of finite length and let $\mathcal L$ be the corresponding even liaison class, i.e. the set of curves $C$ in $\mathbb P^3$ for which $M(C)$ is isomorphic to a shift of $M$ (see \cite{rao}). Which shifts occur, and which contain smooth curves? It is known (see \cite{BM4}) that there is a leftmost shift $M(\delta_0)$ of $M$ that can occur, i.e. $\delta_0$ is such that there is at least one curve $C \in \mathcal L$ with $M(C) = M(\delta_0)$ but there are no curves with  $M(C) = M(\delta_0 +1)$. Denote by $\mathcal L^0$ all such curves, and without loss of generality assume that $M$ has this leftmost shift. The curves in $\mathcal L^0$ are so-called {\it minimal} curves in $\mathcal L$, and they all have the same degree (which is minimal among elements of $\mathcal L$), Hilbert function, and the Hilbert function of the general hyperplane section (see \cite{BBM}). The set of curves for which $M(C) = M(-h)$ ($h \geq 1$) is denoted $\mathcal L^h$, and these do not have to have the same degree, but the Hilbert functions of such curves can be described in terms of the Hilbert functions of elements of $\mathcal L^0$ (cf. \cite{BM4} or \cite{BBM}).

\begin{proposition}\label{11buchs}
There is a smooth arithmetically Buchsbaum curve $C$ of degree 15 and arithmetic genus 25 for which 
\[
\dim [M(C)]_i = 
\left \{
\begin{array}{ll}
1 & \hbox{if } i = 3,4; \\
0 & \hbox{otherwise}.
\end{array}
\right.
\]
This is the smallest degree of any smooth arithmetically Buchsbaum curve in $\mathbb P^3$ {with a Hartshorne-Rao module whose diameter is at least two.
}
\end{proposition}

   \begin{proof} 
    (``Diameter" refers to the number of components from the first to the last nonzero one. ``Dimension" refers to the total dimension of $M(C)$ as a $\mathbb{K}$-vector space.)

    Let $\mathcal L_{1,1}$ be the even liaison class (which in fact is the same as the regular liaison class in this case) of arithmetically Buchsbaum curves with $M(C)$ having diameter and dimension both equal to 2. 
    
    The paper \cite{BM2} uses the language of {\it numerical character} rather than the $h$-vector, but the information is equivalent. Here we will give the necessary results in our $h$-vector notation. See \cite[\S 1]{GM3}  for the dictionary, and see \cite[Notation 1.1]{BM2}  for the notation used in the cited result in \cite{BM2}.

    By \cite[Lemma 4.3]{BM2},  there are no smooth curves in $\mathcal L_{1,1}^0$. These curves all have maximal rank and degree 10. However, \cite[Proposition 4.7]{BM2}  says that in $\mathcal L_{1,1}^1$ there are smooth, irreducible maximal rank curves. The  $h$-vector for the general hyperplane section of any of these smooth curves is $(1,2,3,4,4,1)$, corresponding to the numerical character $(5,5,5,6)$, and the curves have degree 15. This $h$-vector gives the Hilbert function $(1,3,6,10,14, 15, 15 , \dots)$ for the general hyperplane section $\Gamma$ of the curve $C$, and  
    together with the fact that the curves are in the even liaison class $\mathcal L_{1,1}$, we obtain the arithmetic genus 25 using the formula 
\[
g = \sum_{i=1}^\ell [d - h_\Gamma(i)] - K  ,
\]
where $K$ is the vector space dimension of a certain submodule of the Hartshorne-Rao module of $C$, which in our case (because $C$ is arithmetically Buchsbaum) is the entire Hartshorne-Rao module, so $K = 2$ (cf. \cite[Proposition 1.4.2]{MBook}).

    The assertion that this is the smallest degree among arithmetically Buchsbaum smooth curves follows from the Lazarsfeld-Rao property and the explicit construction of minimal elements given in \cite{BM2} (see also \cite{GM1}). If $C \in \mathcal L^0_{1,1}$ then its degree is 10 (\cite[Theorem 3.1]{BM2}), but it cannot be irreducible (\cite[Theorem 2.12]{BM2}). Since there is no reduced, irreducible curve in $\mathcal L_{1,1}$ lying on a surface of degree 4 (\cite[Corollary 2.11 (b)]{BM2}), the Lazarsfeld-Rao property gives that the curve constructed here is the smallest reduced, irreducible curve in $\mathcal L_{1,1}$. But no other arithmetically Buchsbaum liaison class whose module has at least two components (so WLP can fail) contains any curve of degree 15 (use \cite[Theorem 3.1 and Theorem 2.1]{BM2}, as well as the Lazarsfeld-Rao property \cite{BBM}, which did not yet exist in full generality when \cite{BM2} was written).
   \end{proof}

Now we return to Question \ref{Q for curves}. 
We know $d \leq 15$ by Proposition \ref{11buchs}. 

\begin{proposition} \label{at least 10}
        The minimum degree $d$ is at least 10.
\end{proposition}

\begin{proof}
    Let $C$ be a smooth, irreducible curve of degree $d \leq 9$. Let $L$ be a general linear form. Since $C$ is irreducible, its general hyperplane section has UPP, and in particular the $h$-vector of the general hyperplane section is of decreasing type. Clearly it is not of the form $(1,1,\dots,1)$ since $C$ is not a plane curve (which would be ACM). We list the possible $h$-vectors:

    \renewcommand{\labelenumi}{(\arabic{enumi})  }
\renewcommand{\labelenumii}{(\alph{enumii})}

    \begin{enumerate}
        \item \label{3} $(1,2)$;
        \item \label{4} $(1,2,1)$;
        \item \label{5} $(1,2,2)$;
        \item \label{6a} $(1,2,3)$;
        \item \label{6b} $(1,2,2,1)$;
        \item \label{7a} $(1,2,3,1)$;
        \item \label{7b} $(1,2,2,2)$;
        \item \label{8a} $(1,2,3,2)$;
        \item \label{8b} $(1,2,2,2,1)$;
        \item \label{9a} $(1,2,3,3)$;
        \item \label{9b} $(1,2,3,2,1)$;
        \item \label{9c} $(1,2,2,2,2)$.
    \end{enumerate}

    From the exact sequence 
    \begin{equation} \label{usual ses}
    0 \rightarrow [I_C]_{t-1} \stackrel{\times L}{\longrightarrow} [I_C]_t \rightarrow [I_{Z|H}]_t \rightarrow [M(C)]_{t-1} \stackrel{\times L}{\longrightarrow} [M(C)]_t \rightarrow H^1(\mathcal I_Z(t)) \rightarrow \dots
    \end{equation}
    we know that in order for $\times L$ to fail to have maximal rank, it cannot happen that $[I_{Z|H}]_t =0$ and $h^1(\mathcal I_Z(t)) =0$ for the same $t$. This rules out (\ref{3}), (\ref{4}), (\ref{5}), (\ref{6a}), (\ref{7a}), (\ref{8a}), and (\ref{9a}).

    By UPP, if the $h$-vector is that of a complete intersection, then $Z$ must be a complete intersection (Corollary \ref{cor:HF of CI}). Among the remaining cases, this happens for (\ref{6b}), (\ref{8b}) and (\ref{9b}). The socle for such an Artinian ideal is in the last degree. By Theorem \ref{HU fact}, this forces $C$ to lie on a quadric surface in the cases (\ref{6b}) and (\ref{8b}), and we saw in Corollary \ref{curve on quadric} that this forces $M(C)$ to have WLP. As for (\ref{9b}), this forces $C$ itself to be a complete intersection \cite{strano} since the Hilbert function shows that $C$ does not lie on a quadric surface. 

    We are left with only (\ref{7b}) and (\ref{9c}). As above, the socle occurs only in one degree, namely at the end, and so the $h$-vector forces $C$ to lie on a quadric surface. Thus,  we are done.
\end{proof}

We add to Proposition \ref{at least 10} the following partial result.

\begin{proposition} \label{< 15}
    Apart from the following three possible exceptions, the smallest degree of a smooth, irreducible curve $C$ for which $M(C)$ fails WLP is 15. 

    \begin{itemize}
        \item A curve of degree 10 whose general hyperplane section has $h$-vector $(1,2,3,3,1)$;

        \item A curve of degree 13 whose general hyperplane section has $h$-vector $(1,2,3,4,2,1)$;

        \item A curve of degree 14 whose general hyperplane section has $h$-vector $(1,2,3,4,3,1)$.
    \end{itemize}
\end{proposition}

\begin{proof}
    We know from Proposition \ref{11buchs} that a smooth curve of degree 15 exists, so we would like to eliminate curves of degree $\leq 14$. 

     In the proof of Proposition \ref{at least 10} we saw that generic $h$-vectors and $h$-vectors of the type $(1,2,2,\dots,2,1)$ or $(1,2,2,\dots,2,2)$ cannot occur. Also, in the same manner as that used to exclude (\ref{9b}), we can eliminate $(1,2,3,3,2,1)$. We are left with the following list (new numbering):

        \renewcommand{\labelenumi}{(\arabic{enumi})  }

    \begin{enumerate} 
\item \label{10} $(1,2,3,3,1)$;
\item \label{11} $(1,2,3,3,2)$;
\item \label{12} $(1,2,3,3,3)$; 
\item \label{13a} $(1,2,3,3,3,1)$;
\item \label{13b} $(1,2,3,4,2,1)$;
\item \label{14a} $(1,2,3,3,3,2)$;
\item \label{14b} $(1,2,3,4,3,1)$.
    \end{enumerate}

Cases (\ref{10}), (\ref{13b}) and (\ref{14b}) are the open cases listed in the statement of the proposition. It remains to  eliminate (\ref{11}), (\ref{12}), (\ref{13a}) and (\ref{14a}). By \Cref{HU fact}, we will first show that all of these force $C$ to lie on a unique cubic surface, because the socle has to start in degree $\geq 4$ and $Z \subset H$ lies on a unique cubic curve.  This is obvious for (\ref{12}), (\ref{13a}) and (\ref{14a}). For (\ref{11}), since $Z$ has UPP it lies in a complete intersection of type $(3,4)$ (see Remark \ref{use irred} for more details), and so it is linked to a single point. This means that $I_Z$ has three minimal generators, namely in degrees 3, 4 and 5. Then the minimal free resolution (over the polynomial ring $S$ of $H$) must be
\[
0 \rightarrow 
\begin{array}{c}
S(-6)^2 
\end{array} 
\rightarrow
\begin{array}{c}
S(-3) \\
\oplus \\
S(-4) \\
\oplus \\
S(-5)
\end{array}
\rightarrow I_Z \rightarrow 0,
\]
and so the socle {of $M(C)$} is concentrated in degree 4, and hence $[I_Z]_{\leq 3}$ lifts to $I_C$. 

Now, for (\ref{11}), (\ref{12}), (\ref{13a}) and (\ref{14a}) we note, using the sequence (\ref{usual ses}), that the lifting of the cubic in $I_{Z|H}$ gives us injectivity for $\times L \colon [M(C)]_{t-1} \rightarrow [M(C)]_t$ in all degrees where $h^1(\mathcal I_Z(t)) \neq 0$. (For (\ref{13a}) and (\ref{14a}) we also use the fact that there is no new generator in degree 4.) Thus WLP holds for these cases.
\end{proof}

\begin{remark} \label{HU not enough}
We would like to rule out (\ref{10}), (\ref{13b}),  and (\ref{14b}) as possible $h$-vectors of the general hyperplane section of a smooth {\it irreducible} curve of degree $< 15$ for which $M(C)$ fails WLP. 

Note  that if we drop the requirement that $C$ be irreducible, the first of the three open possibilities is the $h$-vector corresponding to the curve in Proposition \ref{all but 2} (when $s=10$), which is smooth if the two extra lines are general; hence, that $h$-vector does occur for some smooth curve for which $M(C)$ fails WLP, but this curve is reducible. The question here is whether an irreducible example exists. 

In fact, in the same way, we can obtain smooth reducible examples for (\ref{13b})  and (\ref{14b}). One can check that taking the union of 10 general lines from the same ruling on a smooth quadric surface and 3 general lines gives $\dim [M(C)]_3 = 32$,  $\dim [M(C)]_4 = 31$, $\dim [I_{Z|H}]_4 = 3$, and $ \dim [I_C]_4 = 1$, so we have a failure of maximal rank for (\ref{13b}) using the sequence (\ref{usual ses}). Similarly, taking the union of 10 general lines from the same ruling on a smooth quadric surface and 4 general lines gives an example of a failure of maximal rank for $(\ref{14b})$.
\end{remark}

\begin{remark} \label{use irred}
    We expand on the previous remark. Here we will show how the assumption of irreducibility can make a difference in our analysis.
    
    Consider curves of degree 15. Proposition \ref{11buchs} gives an example where there is a smooth irreducible curve for which $M(C)$ fails to have WLP. The $h$-vector of its general hyperplane section is $(1,2,3,4,4,1)$. As in Remark \ref{HU not enough}, Theorem \ref{HU fact} is inconclusive with this $h$-vector, and in fact such a curve does exist.

    Now consider the $h$-vector $(1,2,3,4,3,2)$ and suppose $C$ is a curve whose general hyperplane section has this $h$-vector. Does Theorem \ref{HU fact} apply? Here the question of irreducibility of $C$ matters! Suppose first that $C$ is irreducible, so its general hyperplane section $Z$ has the UPP. Then by \cite[Remark 1.2]{MR}, there is an irreducible form among the quartic curves containing $Z$, so $Z$ lies in a complete intersection of two quartics, linking $Z$ to  a single point. Then it is clear that the socle of the Artinian reduction of $S/I_Z$ is concentrated in degree~5. In particular, Theorem \ref{HU fact} gives that the quartics lift to surfaces containing $C$, so $C$ is linked to a line and hence is ACM (and clearly has the same $h$-vector as its Artinian reduction).

    On the other hand, if $C$ is not irreducible, we can take $C$ to be the union of 11 general lines from a ruling of a smooth quadric and four general lines in $\mathbb P^3$. Then one can check, as in Remark \ref{HU not enough}, that $M(C)$ has dimension 40 in degrees 3 and 4, but for a general linear form $L$, $\times L$ is not an isomorphism in these degrees. In fact, the rank is not even submaximal: it is 38. One can check that the general hyperplane section of $C$ has the desired $h$-vector. The difference here is that the quartics containing the general hyperplane section all have a conic factor in common.
\end{remark}

   In the category of smooth, irreducible curves,
   a well-studied family 
   is the class of smooth rational curves.
   This family forms a natural first candidate for the genus part of \Cref{Q for curves}.

   \begin{question}
       If $C$ is a smooth, irreducible rational curve, must $M(C)$ have WLP?
   \end{question}

Although we do not have an answer to this question, we can give a negative answer for curves of arithmetic genus zero more generally, and a positive answer for the general rational curve.

\begin{proposition} \label{arithgen0}
    Let $C_1$ be a smooth curve of type $(7,1)$ on a smooth quadric surface, and let $C_2$ and $C_3$ each be a general line meeting $C_1$ in a point. (In particular, neither $C_2$ nor $C_3$ lies on the quadric containing $C_1$.) Let $C = C_1 \cup C_2 \cup C_3$. Then $C$ has arithmetic genus 0 and degree 10, and $M(C)$ fails to have WLP.
\end{proposition}

\begin{proof}
    Certainly all three components have genus 0. Then the genus of $C_1 \cup C_2$ is $0 + 0 -1 + 1 = 0$, and similarly for $C$. The degree is obvious. The failure of $\times L \colon [M(C)]_2 \rightarrow [M(C)]_3$ to be injective is identical to the situation in Proposition \ref{all but 2}, using (\ref{usual 3}); the key point is that $C$ does not lie on a cubic surface, but its general hyperplane section lies on a cubic curve in the plane.
    
    It remains to check that we expect injectivity from degree 2 to degree 3. We use (\ref{usual 1}) to compute the dimensions:
    \[
\dim [M(C)]_t = \left \{
\begin{array}{ll}
0 & \hbox{if } t=0; \\
7 & \hbox{if } t=1; \\
11 & \hbox{if } t=2; \\
11 & \hbox{if } t=3.
\end{array}
\right.
    \]
    (We suppress the computation for $t \geq 4$ since we do not need it.)
\end{proof}

Passing to smooth rational curves, we now show that a general rational curve $C$ of degree $d$ has $M(C)$ satisfying WLP, even in $\PP^n$.

\begin{proposition}\label{prop:rational curve}
    If $C$ is a general rational curve in $\PP^n$ of degree $d \geq n$ then $M(C)$ has WLP.
\end{proposition}

\begin{proof}
    Note that it makes sense to talk about a general rational curve of degree $d$ since the parameter space is irreducible, {cf.\ \cite[Section 0.4]{Fulton}}. Also, if $d=n$ or $d=n+1$, $M(C)$ has dimension either $0$ or $1$, so the result is trivial.

    Let $H = \PP^{n-1}$ be a hyperplane in $\PP^n$ and let $Z$ be a general set of $d$ points on $H$. By \cite[Theorem 1.6]{BM}, there is a smooth rational curve $C$ in $\PP^n$ whose hyperplane section is $Z$. By generality, for any $t$, either $h^0(\mathcal I_{C|H}(t)) = 0$ or $h^1(\mathcal I_C(t)) = 0$. If $L$ is a linear form defining $H$, then as in Proposition \ref{Prop:GenlWLP} we know that $\times L \colon M(C)_{t-1} \rightarrow M(C)_t$ has maximal rank. By semicontinuity, then, the same is true for a general $L$. Thus for this specific rational curve, $M(C)$ has WLP. But then, again by semicontinuity, the same is true for a general such curve.
\end{proof}

%%%%%%%%%%%%%%%%%%%%%%%%%%%%%%%%%%%%%%%%%%%%%%%%%%%%%%%%%%%%%%%%%%%%%%%

\section{Open questions}

{In this section, we collect the main open questions that arise naturally from the results presented in this paper. While several aspects of the picture are now clearer, a number of cases remain only partially understood. We therefore conclude by formulating the problems that, in our view, deserve further investigation, together with related conjectures.\hfill\break 

In this paper and in the discussion below, we focus on $\PP^3$. However, the problems are meaningful in $\PP^n$ for any $n \ge 3$, and we expect analogous answers.

\textbf{Questions on  unions of lines.} In \Cref{sec: general lines} we considered the case in which $C\subset \PP^3$ is the union of $r$ general lines. We proved that the Hartshorne-Rao module $M(C)$ satisfies SLP in range $\leq 3$ (see \Cref{Prop:GenlWLP}, \Cref{SLP2}, and \Cref{SLP3}).  We believe that this should hold in any range.

\begin{conjecture}\label{conj:gen lines}
    Let $C\subset\PP^3$ be the union of $r$ general lines. Then $M(C)$ has SLP.
\end{conjecture}

There are two main challenges in generalizing our method to higher degrees. In the proof of \Cref{Z2} and \Cref{Z3}, a fundamental step is the specialization of lines into the hyperplane $H=V(L)$. In the cases of $L^2$ and $L^3$, the numerology allowed us to control the socle degree of the specialized scheme. For higher powers of $L$, however, the numerical constraints become more difficult, and the behavior of the socle degree is no longer clear.

A second point is the technical result \Cref{GenFlatFat}. To use the same methodology, the result needs to be extended to higher $m$, which is an interesting problem in its own right. We conjecture the general behavior for $m\geq 4$.

\begin{conjecture}\label{conj:mflatfat}
    Let $Z_{s,m}$ be a general union of $s$ flat fat points of multiplicity $m$ in $\mathbb P^2$. Assume $m \geq 4$ and $s \geq 2m-3$. 
    
    Then $Z_{s,m}$ has generic Hilbert function, i.e.
    \[    h_{Z_{s,m}}(j) = \min \left \{ ms, \binom{j+2}{2} \right \}.
    \]
\end{conjecture}

It is easy to see that for $s\leq 2m-4$, \Cref{conj:mflatfat} fails. Indeed, $Z_{s,m}$ lies on the union of $s$ lines, while a simple computation shows that the predicted Hilbert function of ${Z_{s,m}}$ in degree $s$ is equal to $\binom{s+2}{2}$.

In \Cref{sec4}, we showed how specializing some of the lines of $C$ can change the Lefschetz property of $M(C)$. Specifically, WLP holds when all lines are on a smooth quadric (\Cref{prop:lines quadric}), and it may fail if a few are not (\Cref{all but 2}). A natural question is whether \Cref{prop:lines quadric} extends to SLP. 

\begin{question}
    Let $C\subset Q\subset \PP^3$, where $Q$ is a smooth quadric surface, and $C$ is a union of skew lines in $Q$. Does $M(C)$ have SLP?
\end{question}

As in \Cref{curve on quadric}, a positive answer would also imply that the Hartshorne-Rao module of any curve on a smooth quadric has SLP. 

It is also natural to ask whether allowing one line outside the quadric still preserves SLP, as it does for WLP (\Cref{all but one}), or whether, analogously to the behavior of WLP with two lines outside the quadric, SLP would fail.\hfill\break

\textbf{Questions on smooth irreducible curves.} 

In \Cref{at least 10}, we saw that the smallest possible degree of a smooth irreducible curve $C$ whose Hartshorne-Rao module fails WLP is $10$, and that such examples exist in degree $15$ (Proposition \ref{11buchs}). In particular, \Cref{< 15} rules out all curves in this interval apart from three cases. It is therefore natural to ask whether the Weak Lefschetz Property holds in these remaining cases. 

\begin{question}
Let $C \subset \PP^3$ be a smooth irreducible curve whose general hyperplane section  has one of the following $h$-vectors 
\[
(1,2,3,3,1),\quad (1,2,3,4,2,1),\quad (1,2,3,4,3,1).
\]
Does the Hartshorne-Rao module $M(C)$ satisfy WLP?
\end{question}
These $h$-vectors correspond to curves whose general plane sections are sets of points failing to impose independent conditions on plane curves {of suitable degrees}. 

We briefly examine the case $(1,2,3,3,1)$ as an illustrative example. The general plane section $Z$ of $C$ is $10$ points lying on a unique plane cubic. Notice that, by \eqref{usual ses}, failure of WLP can only happen in the map $$[M(C)]_2\to [M(C)]_3.$$ In particular, $M(C)$ has WLP if and only if the plane cubic containing $Z$ lifts to a cubic surface containing $C$. 

Notice that \Cref{HU fact} does not apply to this case. In fact, this is a sharp case of the classical Laudal’s Lemma: if $\deg C$ were $\geq 11$ instead of $10$, lifting would be guaranteed. 

In Section 5, we showed that, generically, the Hartshorne-Rao module of a nondegenerate rational curve $C\subset \PP^n$ has WLP. { \Cref{arithgen0} provides a cautionary example; nevertheless, we ask:

\begin{question} \label{Qratl}
    If $C$ is an arbitrary smooth rational curve in $\mathbb P^3$, does $M(C)$ have WLP?
\end{question}
}

If Question \ref{Qratl} has a negative answer, it would be desirable to identify the geometric conditions on a smooth rational curve $C$ in $\mathbb P^3$ that ensure that $M(C)$ has WLP. Since every such curve arises as a projection of a rational normal curve in a larger projective space, this leads to the following natural question. 

\begin{question}
    Let $\pi\colon \PP^m\to \PP^n$ be a projection map, and let $C\subset \PP^m$ be smooth and nondegenerate. Assume $M(C)$ has WLP. Under what conditions does $M(\pi(C))$ still have WLP?
\end{question}

}

\section*{ Statements and Declarations}

\noindent{\bf Competing interests:} The authors have no  potential conflicts 
of interest (financial or non-financial) to declare that are relevant to the content of this article.

\end{document}